\RequirePackage{fix-cm}
\documentclass{svjour3}                     % onecolumn (standard format)
\smartqed  % flush right qed marks, e.g. at end of proof
\usepackage{graphicx}
\usepackage{todonotes}
\usepackage{amsbsy}
\usepackage{booktabs}
\usepackage{amsmath,amsfonts,amssymb} %removed amsthm
\usepackage{algpseudocode}
\usepackage{algorithm}
\usepackage{tablefootnote}
\usepackage{subfigure}
\usepackage[title]{appendix}
\usetikzlibrary{arrows,decorations.pathmorphing,backgrounds,positioning,fit,petri}
\usetikzlibrary{decorations.pathreplacing}
\usetikzlibrary{decorations.markings}
\usetikzlibrary{decorations.shapes}
\usetikzlibrary{arrows.meta}
\usetikzlibrary{quotes,angles}
\usetikzlibrary{positioning}
\usetikzlibrary{patterns} 
\usetikzlibrary{chains}
\usetikzlibrary{hobby}
\usepackage{pgfplots}
\usepackage{mathtools}

\newtheorem{assumption}{Assumption}
\newtheorem{condition}{Condition}

\newcommand{\inlinenorm}[1]{ \Vert #1 \Vert}
\newcommand{\norm}[1]{\left\Vert #1 \right\Vert}

\newcommand{\bigpar}[1]{\left( #1 \right)}

\DeclareMathOperator{\gra}{grad}
\DeclareMathOperator{\low}{lower}
\DeclareMathOperator{\upp}{upper}

\DeclareMathOperator{\true}{true}

\DeclareMathOperator{\offo}{OFFO}
\begin{document}

\title{
  A Novel First-order Method with Event-driven Objective Evaluations
}

\titlerunning{Economical Objective Evaluations}        % if too long for running head

\author{Christian Varner         \and
        Vivak Patel %etc.
}

\authorrunning{Varner and Patel} % if too long for running head

\institute{C. Varner \at
              UW Madison Department of Statistics \\
              1220 Medical Sciences Center\\
              1300 University Ave\\
              Madison, WI 53706\\
              \email{cvarner@wisc.edu}           %  \\
           \and
           V. Patel \at
              UW Madison Department of Statistics \\
              1220 Medical Sciences Center\\
              1300 University Ave\\
              Madison, WI 53706\\
              \email{vivak.patel@wisc.edu}
}

\date{Received: date / Accepted: date}

\maketitle

\begin{abstract}
Arising in semi-parametric statistics,
control applications, and as subproblems in global optimization methods,
certain optimization problems can have objective functions requiring 
numerical integration to evaluate, 
yet gradient function evaluations that are relatively cheap.
For such problems, typical optimization methods that require multiple
evaluations of the objective for each new iterate become computationally expensive. 
In light of this, optimization methods that avoid objective function evaluations 
are attractive, yet  we show anti-convergence behavior for these methods on 
the problem class of interest.
To address this gap, 
we develop a novel gradient algorithm that only evaluates the objective function 
when specific events are triggered and propose a step size scheme that greedily 
takes advantage of the properties induced by these triggering events. 
We prove that our methodology has global convergence guarantees under 
the most general smoothness conditions, 
and show through extensive numerical results that our method performs
favorably on optimization problems arising in semi-parametric statistics.

\keywords{Local Lipschitz Smoothness \and Adaptive Step Scheme \and Novel Line Search Strategy
\and First-Order Gradient Method}
\subclass{90C30 \and 65K05 \and 68T09}
\end{abstract}

\section{Introduction}         
\label{sec:intro}
In semi-parametric statistics \cite{wedderburn1974quasilikelihood,mccullagh1989glm}, 
engineering and optimal control applications \cite{lloyd2020using,Labbe2024gibo},
and global optimization schemes \cite{Labbe2024gibo}, optimization problems can take the form
\begin{equation} \label{eq-integral-objective-function}
    \min_{\theta \in \mathbb{R}^n} F(\theta) = 
    \min_{\theta \in \mathbb{R}^n} \sum_{i=1}^m \int_{\mathcal{S}_i(\theta)} f_i(x) dx,
\end{equation}
where $f_i: \mathbb{R}^p \to \mathbb{R}$ is a known (possibly) non-convex function
and $\mathcal{S}_i:\mathbb{R}^n \to \mathbb{R}^p$ is a known set-valued map for $i=1,\ldots,m$.
In the aforementioned application areas, 
the optimization problem's objective function, $F(\theta)$, 
typically does not have a closed form and must be evaluated by expensive numerical 
integration schemes; yet, the gradient can be evaluated relatively cheaply by
the fundamental theorem of calculus.   

As a simple, contrived example, globally optimizing the function 
\begin{equation}
    f_1(\theta) = (1 + \theta^8)^{1/2} + (2 - \theta^6)^{1/3},
\end{equation}
using \cite[Algorithm 2]{Labbe2024gibo}, results in an optimization problem of 
the form in (\ref{eq-integral-objective-function}),
\begin{equation} \label{eq-simple-example}
    \min_{\theta} \int_{\theta-w}^{\theta+w} f_1(x) dx,
\end{equation}
where $w > 0$. 
The resulting optimization problem's objective function does not have a closed 
form anti-derivative \cite[Page 10-11]{yadav2012}; 
yet, the gradient function has a closed form of $f_1(\theta+w) - f_1(\theta-w)$.
In other words, the objective function is expensive to evaluate, while the gradient 
function is relatively cheap to evaluate. 
Real examples of optimization problems from statistics that also satisfy this computational 
pattern will be given in Section \ref{sec:experiments}.

For such optimization problems, objective-evaluation heavy optimization methods can
be computationally expensive (e.g., line-search-type methods), 
making objective-function-free-optimization (OFFO) algorithms attractive alternatives.
OFFO methods \cite{grapiglia2022Adaptive,gratton2022Convergence,malitsky2023adaptive,zhou2024AdaBB,oikonomidis2024Adaptive,wang23convergenceadagrad} 
dispense with objective function evaluations by introducing (adaptive) step size schemes, 
while also enjoying global convergence guarantees under a variety of conditions,
such as convexity and variations on Lipschitz smoothness (see Table \ref{table-offo-examples}).
However, these conditions are not satisfied by our simple example (\ref{eq-simple-example})
nor those we consider in Section \ref{sec:experiments}, which are instead 
\emph{non-convex and locally Lipschitz smooth}.

\begin{table}[ht]
    \centering
    \caption{OFFO methods and references. 
    Convexity indicates if the convergence theory for the
    algorithm requires convexity of the objective.
    Smoothness indicates the continuity assumptions on the function, gradient, and/or
    hessian.
    We use (GLO) for globally Lipschitz continuity (see Definition 1 \cite{patel2024recentadvancesnonconvexsmoothness}),  
    ($L_0, L_1, \rho$) for $\rho-$order Lipschitz continuity (see Definition 3 \cite{li2023Convex}),
    (LOC) for locally Lipschitz continuity (see Definition \ref{def-locally-lip}), 
    ($\alpha-$H\"{O}L) for $\alpha$ locally H\"{o}lder continuity (see Section 2.1 \cite{oikonomidis2024Adaptive}),
    and $(r, \ell)$ to indicate generalized $\ell-$continuity (See Definition 2 \cite{li2023Convex}). 
    We will prepend F-, G-, H-, and p- to the codes above to indicate whether the function,
    gradient, hessian, or $p^{th}$ derivative satisfies these smoothness conditions, 
    respectively.}
    \label{table-offo-examples}
    \footnotesize
    \begin{tabular}{p{2in} c c} \toprule
    \textbf{Method} & \textbf{Convexity} & \textbf{Smoothness}  \\ \midrule
      Constant Step Size \cite{li2023Convex} & Yes & G-$(r, \ell)$\\
      Nesterov's Accelerated Method \cite{li2023Convex} & Yes & G-$(L_0, L_1, \rho),~\rho \in [0, 2)$\\
      Barzilai-Borwein Methods \cite{barzilai1988Twopoint} & 
      Yes\tablefootnote{Only for strongly convex quadratics \cite{raydan1993,dai2002RLinear,burdakov2019Stabilized,zhou2024AdaBB}.} & G-(GLO) \\
      Adaptive Gradient w/o Descent \cite{malitsky2023adaptive,malitsky2020Adaptive} & Yes & G-(LOC) \\
      Adaptive Gradient w/o Descent \cite{oikonomidis2024Adaptive} & Yes & G-($\alpha$-H\"{O}L) \\
      Polyak Step Size \cite{vankov2024optimizingl0l1smoothfunctions} & Yes & G-$(L_0, L_1, 1)$ \\ \midrule
      Diminishing Step Size \cite{patel2024Gradient} & No & G-(LOC)\tablefootnote{Under a bounded iterate assumption \cite{patel2024Gradient}.}\\
      Constant Step Size \cite{li2023Convex} & No & G-$(L_0, L_1, \rho),~\rho \in [0, 2)$\\
      Negative Curvature Method \cite{curtis2019Exploiting} & No & G-(GLO), H-(GLO)\\
    Gradient Clipping \cite{zhang2020firstorder,zhang2020Improved,vankov2024optimizingl0l1smoothfunctions} & No & G-$(L_0, L_1, 1)$ \\
      Normalized Gradient Descent \cite{hubler2023parameteragnosticoptimizationrelaxedsmoothness,vankov2024optimizingl0l1smoothfunctions} & No & G-$(L_0, L_1, 1)$ \\
      Normalized Gradient Descent \cite{chen2023generalized} & No & G-$(L_0, L_1, \rho),~\rho \in [0,1]$ \\
      Weighted Gradient-Norm Damping \cite{wu2020Wngrad} & No & G-(GLO) \\
      Adaptively Scaled Trust Region \cite{grapiglia2022Adaptive} & No & G-(GLO) \\
      Adaptively Scaled Trust Region \cite{gratton2022Convergence} & No & p-(GLO)
   \\ \bottomrule
    \end{tabular}
\end{table}

A natural question is: can OFFO methods be extended to the non-convex and 
locally Lipschitz smooth case, and can they be successful in \emph{minimizing} objective
functions in this class?
By building on a string of anti-convergence results 
\cite{vavasis1993blackbox,cartis2022Evaluationa,patel2024Gradient,varner2024Challenges},
we construct generic, non-convex and locally Lipschitz smooth objective functions 
on which OFFO methods will generate a sequence of diverging iterates for which 
the optimality gap diverges (see Section \ref{sec:counter}).
Furthermore, using our construction, we show an extensive list of OFFO methods 
also fail to drive
the norm gradient to $0$ (see Table \ref{table-counter-examples-divergence} in Section
\ref{sec:counter}). This strongly affirms that
such methods \emph{do not} readily extend to these examples.

To summarize, our problem class of interest has a computational pattern that 
is unfavorable to objective-evaluation heavy methods, and where OFFO methods lack
convergence guarantees.
We address this gap by developing a novel optimization method that 
adaptively determines when to evaluate the objective function to ensure 
economical objective function evaluations; is guaranteed 
to keep the optimality gap bounded thereby avoiding the pitfalls of OFFO methods; 
and enjoys global convergence guarantees under general conditions. 
We accomplish this by structuring our algorithm around three novel elements: 
first, we operationalize the theoretical tool of triggering events to signal 
when to check the objective \cite{patel2021Stochastic,patel2022Global,patel2024Gradient};
second, we introduce an evolution of Armijo's line search \cite{armijo1966Minimization} 
to generalize across multiple gradient iterations;
third, we develop an adaptive step size using the two previous elements, 
alleviating the necessity for frequent tuning.
In Section \ref{sec:theory}, we show our method is globally convergent under
realistic assumption of our problem class.
In Section \ref{sec:convergence-rate}, we develop global and local rate-of-convergence 
results under several scenarios. 
In Section \ref{sec:experiments}, we show that our method is computational practical
and reliable against alternatives on instances of quasi-likelihood estimation 
\cite{wedderburn1974quasilikelihood}.

The paper is organized as follows: 
in section \ref{sec:problem-formulation}, we formally introduce our setting;
in section \ref{sec:counter}, we show our anti-convergence results for OFFO algorithms;
in section \ref{sec:algorithm}, we detail our novel algorithm;
in section \ref{sec:theory}, we perform a global convergence analysis;
in section \ref{sec:convergence-rate}, we perform rate-of-convergence analyses;
in section \ref{sec:experiments}, we discuss numerical experiments;
and, in section \ref{sec:conclusion}, we conclude.

\section{Problem Formulation} \label{sec:problem-formulation}
While we are motivated by problems of the form (\ref{eq-integral-objective-function}), 
we consider a broader class of optimization problems in the development and analysis 
of our procedure.
Namely, we consider 
\begin{equation} \label{eqn-problem}
    \min_{\theta\in\mathbb{R}^n} F(\theta),
\end{equation}
where $F(\theta)$ satisfies the following two assumptions.
\begin{assumption} 
\label{as:lower-bounded}
    The objective function, $F(\theta)$ is continuous and is lower bounded by $F^* > -\infty$.
\end{assumption}

\begin{assumption} 
\label{as:locally-lipschitz-cont}
    The gradient, $\dot F(\theta) = \nabla_\psi F(\psi) \vert_{\psi=\theta}$, 
    is locally Lipschitz continuous.
\end{assumption}
Recall, local Lipschitz continuity is defined as follows.
\begin{definition} \label{def-locally-lip}
    A continuous function $G : \mathbb{R}^n \to \mathbb{R}^d$ is locally Lipschitz continuous if, for every $\theta \in \mathbb{R}^n$, 
    there exists an open ball around $\theta$, denoted by $B(\theta)$,\footnote{This definition can be extended to arbitrary compact sets using sequential compactness.} 
    and there exists a constant $\mathcal{L}(B(\theta)) \geq 0$ such that
    \begin{equation}
        \inlinenorm{G(\theta_1) - G(\theta_2)}_2 \leq \mathcal{L}(B(\theta)) \inlinenorm{\theta_1 - \theta_2}_2,~\forall \theta_1, \theta_2 \in B(\theta).
    \end{equation}
\end{definition}
\begin{remark}
    For a differentiable function $F : \mathbb{R}^n \to \mathbb{R}$, by setting $G = \dot F$ 
    in Definition \ref{def-locally-lip} we arrive at the concept of a locally Lipschitz 
    continuous gradient, or a locally Lipschitz smooth objective.
    We refer the reader to \cite{patel2024recentadvancesnonconvexsmoothness} for a brief 
    overview of recent developments of smoothness conditions. 
\end{remark}

\section{Anti-Convergence Results for OFFO}
\label{sec:counter}
We now consider the worst-case behavior of OFFO methods
when $F(x)$ satisfies Assumptions \ref{as:lower-bounded} and \ref{as:locally-lipschitz-cont}.
The anti-convergence analysis%
\footnote{A related line of work is in tight-lower-bound analyses, which show that a
method will remain away from a stationary point for continuously differentiable 
functions for a certain amount of time \cite{vavasis1993blackbox,cartis2022Evaluationa}. 
Anti-convergence results take analyses further by considering
the stricter class of functions which result in optimality gap divergence as well.}
of OFFO methods traces back to the well known example of applying constant
step size gradient descent to quadratic functions with step sizes larger than 
a multiple of the inverse Lipschitz constant \cite[Exercise 1.2.6b]{bertsekas1999}.
Despite this rich history, anti-convergence analyses for OFFO methods are difficult, 
as evidenced by the elusive construction of a provable divergence example for the 
Barzilai-Borwein method.%
\footnote{To the best of our knowledge, our recent work constructed such 
an example \cite{varner2024Challenges}.}
A recent advancement 
showed that gradient descent with diminishing step sizes applied to a carefully 
constructed objective function exhibits anti-convergence: 
the optimality gap diverges while the gradient norm remains bounded from zero 
\cite[\S 4]{patel2024Gradient}.
Here, we provide a generalization of these results that can be applied broadly to the
class of OFFO methods.

\subsection{Mathematical Model for OFFO Methods} \label{subsec:offo-model}

\begin{table}[ht] 
    \centering
    \caption{
        Optimization algorithms that satisfy Condition \ref{condition-output-same}, and the sequences $\{d_k : k + 1 \in \mathbb{N}\}$ and 
        $\{\theta_k : k + 1 \in \mathbb{N}\}$ such that Condition \ref{condition-iterate-diverging} is satisfied.
        The only algorithm in the list below with $\psi_k \not = \theta_k$ is Nesterov's Acceleration Method \cite{chen2023generalized}, therefore
        we abstain from adding a column (see \cite{varner2024Challenges} for details).
        It is assumed the empty sum is $0$, and that $\theta_0 = 0$ for all the algorithm listed below.
        All extra symbols beyond the ones explicitly given are parameters of the algorithm that the user must specify and remain constant throughout the algorithm.
    }
    \label{table-counter-examples-divergence}
    \footnotesize
    \begin{tabular}{p{2in} c c} \toprule
    \textbf{Method} & $-d_k$ & $\theta_k$ \\ \midrule
        Diminishing Step Size \cite{patel2024Gradient}\tablefootnote{Result first appeared in 
        \cite{patel2024Gradient}.} & $1$ & $\sum_{i=0}^{k-1}\alpha_k$ \\
        Constant Step Size \cite{li2023Convex} & $1$ & $k\alpha$ \\
        Nesterov's Acceleration Method \cite{li2023Convex} & $1$ & \cite[Eq. 3.12]{varner2024Challenges}\\
        Barzilai-Borwein Methods \cite{barzilai1988Twopoint} & $(1/2)^k$ & $\alpha_0 k$ \\
        Adaptive Gradient w/o Descent \cite{malitsky2020Adaptive} & $\left( \sqrt{5}/(\sqrt{5} - 1)\right)^{k}$ & $\alpha_0 \sum_{j = 0}^{k-1} (\sqrt{5}/2)^{j}$ \\
        Normalized Gradient Descent \cite{chen2023generalized} & $1$ & $k\alpha$ \\
        Weighted Gradient-Norm Damping \cite{wu2020Wngrad} & $1$ & \cite[Prop. 3.10]{varner2024Challenges}\\
        Adaptively Scaled Trust Region \cite{gratton2022firstorder} & $1$ & $\sum_{j=0}^{k-1} (\xi + j + 1)^{\mu}$\\
        $(L_0, L_1)-$Step Size \cite[Eq. 3.2]{vankov2024optimizingl0l1smoothfunctions} & $1$ & $\frac{k \log(1 + (L_1/(L_0 + L_1)))}{L_1}$\\
        $(L_0, L_1)-$Step Size \cite[Eq. 3.5]{vankov2024optimizingl0l1smoothfunctions} & $1$ & $\frac{k}{L_0 + (3/2)L_1}$ \\
        $(L_0, L_1)-$Step Size \cite[Eq. 3.6]{vankov2024optimizingl0l1smoothfunctions} & $1$ & $k\min\{1/(2L_0), 1/(3L_1)\}$
    %Diminishing Step-Size & $k$ & $k$ & $1$ & $\sum_{i=0}^{k-1} \alpha_k$ \\
    %Constant Step-Size & $k$ & $k$ & $1$ & $k\alpha$ \\
    %Barzilai-Borwein Methods & $k$ & $k$ & $(\frac{1}{2})^k$ & $\alpha_0 k$ \\
    %Nesterov's Acceleration Method & $2k$ & $2k + 1$ & $1$ & \cite[Eq. 3.12]{varner2024Challenges}\\ % TODO
    %Negative Curvature Method & $k$ & $k$ & $1$ & $\alpha k$ \\
    %Lipschitz Approximation & $k$ & $k$ & $\left( \frac{\sqrt{5}}{\sqrt{5} + 1} \right)^{k}$ & $\alpha_0 \sum_{j = 0}^{k-1} (\sqrt{5}/2)^{j+1}$\\
    %Weighted Gradient-Norm Damping & $k$ & $k$ & $1$ & \cite[Prop. 3.10]{varner2024Challenges}\\
    %Adaptively Scaled Trust Region & $k$ & $k$ & $1$ & $\sum_{j=0}^{k-1} (\xi + j + 1)^{\mu}$ \\ 
    %$(L_0, L_1)-$Step Size \cite[Eq. 3.2]{vankov2024optimizingl0l1smoothfunctions} & $k$ & $k$ & $1$ & $\frac{k \log(1 + (L_1/(L_0 + L_1)))}{L_1}$ \\
    %$(L_0, L_1)-$Step Size \cite[Eq. 3.5]{vankov2024optimizingl0l1smoothfunctions} & $k$ & $k$ & $1$ & $\frac{k}{L_0 + (3/2)L_1}$ \\ 
    %$(L_0, L_1)-$Step Size \cite[Eq. 3.6]{vankov2024optimizingl0l1smoothfunctions} & $k$ & $k$ & $1$ & $\min\{1/(2L_0), 1/(3L_1)\}$ \\
    \\ \bottomrule
    \end{tabular}
\end{table}

We introduce a model of OFFO algorithms in the \textit{one-dimensional parameter case}
(i.e., $n=1$).%
\footnote{We plan to extend this model to multi-dimensional parameters in 
future work.}
Given 
$\theta \in \mathbb{R}$ (current iterate), 
$\psi \in \mathbb{R}$ (ancillary iterate, see examples), 
$\mu \in \mathbb{R}^m$ (state variables, see examples), 
$d \in \mathbb{R}$ (gradient information), 
$k+1 \in \mathbb{N}$,
we model the update of an OFFO method by 
\begin{equation}\label{eq-deterministic-offo-method}
    \mathcal{A}_{\offo}(k, \theta, \psi, \mu, d) = \theta_{+}, \psi_{+}, \mu_{+},
\end{equation}%
where $\theta_{+} \in \mathbb{R}$ is the next iterate;
$\psi_{+} \in \mathbb{R}$ is the next ancillary iterate; 
and $\mathcal{\mu}_{+} \in \mathbb{R}^m$ are updated state variables.
To contextualize this, we provide three examples of common algorithms written in 
this framework.
\begin{example}[Gradient Descent with Diminishing Step Size \cite{bertsekas1999}]
    Gradient descent with diminishing step sizes generates iterates, $\lbrace 
    \theta_k : k+1 \in \mathbb{N} \rbrace$, using a
    pre-specified step size sequence $\{\mu_k : k + 1 \in \mathbb{N}\}$, 
    and the recursion $\theta_{k+1} = \theta_k - \mu_k \dot F(\theta_k)$. 
    In terms of (\ref{eq-deterministic-offo-method}),
    the method's update procedure is 
    \begin{equation}
        \mathcal{A}^{\mathrm{diminishing}}_{\offo}(k, \theta, \theta, \mu_k, \dot F(\theta))
        = \theta - \mu_k \dot F(\theta), \theta - \mu_k \dot F(\theta), \mu_{k+1}.
    \end{equation}
\end{example}

\begin{example}[WNGrad \cite{wu2020Wngrad}]
    The WNGrad algorithm from \cite{wu2020Wngrad} generates iterates, $\lbrace 
    \theta_k : k+1 \in \mathbb{N} \rbrace$, and step sizes, $\lbrace \mu_{k}: 
    k+1 \in \mathbb{N} \rbrace$ by 
    \begin{equation}
        \begin{cases}
            \theta_{k+1} &= \theta_k - \frac{1}{\mu_k} \dot{F}(\theta_k) \\
            \mu_{k+1} &= \mu_k + \frac{1}{\mu_k}\norm{\dot F(\theta_k)}_2^2.
        \end{cases}
    \end{equation}
    In terms of (\ref{eq-deterministic-offo-method}), 
    the method's update procedure is 
    \begin{align}
    &\mathcal{A}^{\mathrm{WNGrad}}_{\offo}(k, \theta, \theta, \mu, \dot F(\theta))
    = \theta - \frac{1}{\mu} \dot F(\theta), 
    \theta - \frac{1}{\mu} \dot F(\theta), 
    \mu + \frac{\inlinenorm{\dot F(\theta)}_2^2}{\mu}.
\end{align}
\end{example}

\begin{example}[Barzilai-Borwein \cite{barzilai1988Twopoint}]
    The long form of the Barzilai-Borwein method generates iterates, 
    $\lbrace \theta_{k}: k+1 \in \mathbb{N} \rbrace$, and step sizes,
    $\lbrace \mu_k : k+1 \in \mathbb{N} \rbrace$, by 
    \begin{equation}
        \begin{cases}
            \theta_{k+1} &= \theta_k - \mu_k \dot F(\theta_k) \\ 
            \mu_{k+1} & = \frac{\norm{\theta_{k+1} - \theta_k}_2^2}{
                (\theta_{k+1} - \theta_{k})^\intercal (\dot F(\theta_{k+1}) - \dot F(\theta_k))
            } = \frac{\norm{\mu_k \dot F(\theta_k)}_2^2}{-\mu_k \dot F(\theta_k)^\intercal
            \left(\dot F(\theta_{k+1}) - \dot F(\theta_k)\right)}.
        \end{cases}
    \end{equation}
    In terms of \eqref{eq-deterministic-offo-method},
    \begin{equation}
        \begin{aligned} 
       &\mathcal{A}^{\mathrm{BB}}_{\offo}(k, \theta, \theta, \mu, \dot F(\theta))\\
       & = \theta - \mu \dot F(\theta), \theta - \mu \dot F(\theta), 
        \frac{\norm{\mu \dot F(\theta)}_2^2}{-\mu \dot F(\theta)^\intercal
        \left(\dot F(\theta - \mu \dot F(\theta)) - \dot F(\theta)\right)}.
        \end{aligned}
    \end{equation}
\end{example}

As the above examples show, we can readily capture OFFO methods using 
\eqref{eq-deterministic-offo-method}; 
however, we can capture many other methods using this model. 
To restrict our discussion to OFFO methods, we consider two conditions that we 
argue are reasonable for OFFO methods to satisfy, and which are satisfied by 
the OFFO methods listed in Table \ref{table-counter-examples-divergence}.
The first condition effectively precludes the use of objective function 
information.

\begin{condition} \label{condition-output-same}
    Let $\mathcal{A}_{\offo}$ be a model for an OFFO method.
    Let $\theta \in \mathbb{R}, \psi \in \mathbb{R}, \mu \in \mathbb{R}^m, k + 1 \in \mathbb{N}$. 
    For any $H_1, H_2 : \mathbb{R} \to \mathbb{R}$ such that $H_1(\psi) = H_2(\psi)$,
    \begin{equation}
        \mathcal{A}_{\offo}(k, \theta, \psi, \mu, H_1(\psi)) = 
        \mathcal{A}_{\offo}(k, \theta, \psi, \mu, H_2(\psi)).
    \end{equation}
\end{condition}

As motivation for the next condition, note that OFFO methods typically have guarantees 
for minimizing a convex objective function whose gradient function is globally 
Lipschitz continuous (with appropriately chosen parameters).%
\footnote{Nearly every method in Table \ref{table-counter-examples-divergence} has 
such a result, if not a more general result that does not require convexity.
A notable exception is the Barzilai-Borwein method for which it has only been shown
to globally converge for strictly convex quadratics \cite{burdakov2019Stabilized}.}
For instance, an OFFO method should be able to minimize a convex, quadratic 
function. As another example, an OFFO method should be able to minimize the 
following convex function whose derivative is globally Lipschitz continuous:
\begin{equation} \label{eq-function-for-cond2}
    F(\theta) = 
    \begin{cases}
        \exp(-\theta) & \theta > 0 \\
        0.5(\theta - 1)^2 + .5 & \theta \leq 0.
    \end{cases}
\end{equation}
For this example, the OFFO method will generate a diverging sequence of iterates, 
which, ideally, will be strictly increasing.\footnote{Indeed, if the method is
gradient related then the iterates will strictly increase.}
These properties are desirable in constructing our anti-convergence function, 
and in view of \eqref{eq-function-for-cond2} are not unreasonable. Therefore, we
encapsulate this behavior in our next condition below.

\begin{condition} \label{condition-iterate-diverging}
    Let $\mathcal{A}_{\offo}$ be a model for an OFFO algorithm.
    Let $\theta_0 \in \mathbb{R}, \psi_0 \in \mathbb{R}, \mu_0 \in \mathbb{R}^m$.
    There exists a sequence $\{d'_k : k + 1 \in \mathbb{N}\} \subset \mathbb{R}$, 
    for which $\{(\theta_k, \psi_k, \mu_k) : k + 1 \in \mathbb{N}\}$, defined by
    \begin{equation}
        \theta_{k+1}, \psi_{k+1}, \mu_{k+1} = \mathcal{A}_{\offo}(k, \theta_k, \psi_k, d'_k),
    \end{equation}
    has the property that $\{\theta_k : k + 1 \in \mathbb{N}\}$ and $\{\psi_k : k + 1 \in \mathbb{N}\}$
    are unbounded strictly increasing sequences.
\end{condition}

\subsection{Objective Function Construction}

We construct an objective function to demonstrate the anti-convergence behavior of 
$\mathcal{A}_{\offo}$.
We use a building block function defined on a finite 
interval for which the function's values and all one-sided 
derivatives can be controlled at the interval's end points, and which is 
lower bounded on the interval to satisfy Assumption \ref{as:lower-bounded}.
A first attempt at constructing such a function might consider polynomial 
interpolation (as used in \cite{vavasis1993blackbox,cartis2022Evaluationa}).
However, the polynomial interpolant quickly becomes unmanageable when we 
try to control all derivatives at the end points and ensure a lower bound. 

Extending a piece-wise function developed in \cite{patel2024Gradient}, 
we use the following building block function, which we term Frankenstein's 
function.
Let $m \in \mathbb{R}_{>0}$,
let $d, \delta \in \mathbb{R}$,
let $a = \max(|d|, 1)$ and $b = \max(|\delta|, 1)$.
With illustrations in Figure \ref{fig-component-function-examples}, define 

\begin{small}
\begin{align} \label{eq:component-functions}
    f(&\theta; m, d, \delta) = \\
    &\begin{cases}
        -d \theta & \theta \in (0, \frac{m}{16a}) \\
        \frac{8da}{m} (\theta - \frac{m}{8a})^2 - \frac{3m}{32}\frac{d}{a} & \theta \in [\frac{m}{16a}, \frac{m}{8a}) \\
        - \frac{3m}{32}\frac{d}{a} & \theta \in [\frac{m}{8a}, \frac{m}{8}] \\
        \frac{8}{m} (\theta-\frac{m}{8})^2 - \frac{3m}{32}\frac{d}{a} & \theta \in (\frac{m}{8}, \frac{3m}{16}) \\
        \frac{-5m}{16}\exp\left( \frac{5/16}{\theta/m - 1/2} + 1 \right) + \frac{11m}{32} - \frac{3md}{32a} & \theta \in [\frac{3m}{16}, \frac{m}{2})\\
        \frac{11m}{32} - \frac{3md}{32a} & \theta = \frac{m}{2} \\
        \frac{5m}{16}\exp\left(\frac{-5/16}{\theta/m - 1/2} + 1\right) + \frac{11m}{32} - \frac{3md}{32a} & \theta \in (\frac{m}{2}, \frac{13m}{16}) \\
        \frac{-8}{m}(\theta-\frac{7m}{8})^2 + \frac{22m}{32} - \frac{3md}{32a} & \theta \in [\frac{13m}{16}, \frac{14m}{16}) \\
        \frac{22m}{32} - \frac{3md}{32a} & \theta \in [\frac{14m}{16}, \frac{(16b-2)m}{16b}] \\
        \frac{-8\delta b}{m}(\theta - \frac{(16b-2)m}{16b})^2 + \frac{22m}{32} - \frac{3md}{32a}  & \theta \in (\frac{(16b-2)m}{16b}, \frac{(16b-1)m}{16b}) \\
        -\delta\theta + \frac{(32b-3)\delta m}{32b} + \frac{22m}{32} - \frac{3md}{32a} & \theta \in [\frac{(16b-1)m}{16b}, m].
    \end{cases}
\end{align}
\end{small}

\begin{figure}[h]
    \centering
    % slopes that are (1, 1)
\begin{tikzpicture}[domain=0:1,scale=3]
    \draw[very thin,color=gray] (-0.1,-.125) grid (.99,.99);

    \draw[->] (-0.1, 0) -- (1.1, 0) node[right] {$x$};
    \draw[->] (0, -0.125) -- (0, 1.1) node[above] {$f(\theta; 1, 1, 1)$}; 

    \draw[thick,domain=0:0.0625,color=blue] plot (\x, -\x);
    \draw[thick,domain=0.0625:0.125,color=purple] plot (\x, {8*(\x-(1/8))^2 - (3/32)});
    \draw[thick,domain=0.125:0.125,color=red] plot (\x, {-3/32});
    \draw[thick,domain=0.125:0.1875] plot (\x, {8*(\x-(1/8))^2-(3/32)});
    \draw[thick,domain=0.1875:0.499] plot (\x, {(-5/16)*exp((5/16)/(\x-.5) + 1) + (11/32) - (3/32)});
    \draw[thick,domain=0.499:0.5001] plot (\x, {(11/32) - (3/32)});
    \draw[thick,domain=0.5001:0.8125] plot (\x, {(5/16)*exp(-(5/16)/(\x-.5) + 1) + (11/32) - (3/32)});
    \draw[thick,domain=0.8125:0.875] plot (\x, {-8*(\x-(7/8))^2 + (22/32) - (3/32)});
    \draw[thick,domain=0.875:0.875,color=red] plot (\x, {(22/32)-(3/32)});
    \draw[thick,domain=0.875:0.9375,color=purple] plot (\x, {-8*(\x-(14/16))^2+(22/32) - (3/32)});
    \draw[thick,domain=0.9375:1,color=blue] plot (\x, {-\x + (29/32) + (22/32) - (3/32)});
\end{tikzpicture}
    \begin{tikzpicture}[domain=0:1,scale=3]
    \draw[very thin,color=gray] (-0.1,-.125) grid (.99,.99);

    \draw[->] (-0.1, 0) -- (1.1, 0) node[right] {$x$};
    \draw[->] (0, -0.125) -- (0, 1.1) node[above] {$f(\theta; 1, 2, 3)$};   

    % a = 2, c = 3
    \draw[thick,domain=0:0.03125,color=blue] plot (\x, -2 * \x);
    \draw[thick,domain=0.03125:0.0625,color=purple] plot (\x, {32*(\x-(1/16))^2 - (3/32)});
    \draw[thick,domain=0.0625:0.125,color=red] plot (\x, {-3/32});
    \draw[thick,domain=0.125:0.1875] plot (\x, {8*(\x-(1/8))^2-(3/32)});
    \draw[thick,domain=0.1875:0.499] plot (\x, {(-5/16)*exp((5/16)/(\x-.5) + 1) + (11/32) - (3/32)});
    \draw[thick,domain=0.499:0.5001] plot (\x, {(11/32) - (3/32)});
    \draw[thick,domain=0.5001:0.8125] plot (\x, {(5/16)*exp(-(5/16)/(\x-.5) + 1) + (11/32) - (3/32)});
    \draw[thick,domain=0.8125:0.875] plot (\x, {-8*(\x-(7/8))^2 + (22/32) - (3/32)});
    \draw[thick,domain=0.875:(46/48),color=red] plot (\x, {(22/32)-(3/32)});
    \draw[thick,domain=(46/48):(47/48),color=purple] plot (\x, {-8*3*3*(\x-(46/48))^2+(22/32) - (3/32)});
    \draw[thick,domain=(47/48):1,color=blue] plot (\x, {-3 * \x + ((32*3-3)*3/(32*3)) + (22/32) - (3/32)});
\end{tikzpicture}
    \begin{tikzpicture}[domain=0:1,scale=3]
    \draw[very thin,color=gray] (-0.1,-.125) grid (.99,.99);

    \draw[->] (-0.1, 0) -- (1.1, 0) node[right] {$x$};
    \draw[->] (0, -0.125) -- (0, 1.1) node[above] {$f(\theta; 1, -2, -3)$};   

    % a = -2, c = -3
    \draw[thick,domain=0:0.03125,color=blue] plot (\x, 2 * \x);
    \draw[thick,domain=0.03125:0.0625,color=purple] plot (\x, {8*(-2)*2*(\x-(1/16))^2 + (3/32)});
    \draw[thick,domain=0.0625:0.125,color=red] plot (\x, {3/32});
    \draw[thick,domain=0.125:0.1875] plot (\x, {8*(\x-(1/8))^2+(3/32)});
    \draw[thick,domain=0.1875:0.499] plot (\x, {(-5/16)*exp((5/16)/(\x-.5) + 1) + (11/32) + (3/32)});
    \draw[thick,domain=0.499:0.5001] plot (\x, {(11/32) + (3/32)});
    \draw[thick,domain=0.5001:0.8125] plot (\x, {(5/16)*exp(-(5/16)/(\x-.5) + 1) + (11/32) + (3/32)});
    \draw[thick,domain=0.8125:0.875] plot (\x, {-8*(\x-(7/8))^2 + (22/32) + (3/32)});
    \draw[thick,domain=0.875:(46/48),color=red] plot (\x, {(22/32)+(3/32)});
    \draw[thick,domain=(46/48):(47/48),color=purple] plot (\x, {-8*(-3)*3*(\x-(46/48))^2+(22/32) + (3/32)});
    \draw[thick,domain=(47/48):1,color=blue] plot (\x, {3 * \x + ((32*3-3)*(-3)/(32*3)) + (22/32) + (3/32)});
\end{tikzpicture}
    \begin{tikzpicture}[domain=0:1,scale=3]
    \draw[very thin,color=gray] (-0.1,-.125) grid (.99,.99);

    \draw[->] (-0.1, 0) -- (1.1, 0) node[right] {$x$};
    \draw[->] (0, -0.125) -- (0, 1.1) node[above] {$f(\theta; 1, 0.25, 0.75)$};   

    % d = 0.25, c = 0.75
    \draw[thick,domain=0:(1/16),color=blue] plot (\x, -0.25 * \x);
    \draw[thick,domain=(1/16):(1/8),color=purple] plot (\x, {8*(0.25)*(\x-(1/8))^2 - (0.25)*(3/32)});
    \draw[thick,domain=(1/8):(1/8),color=red] plot (\x, {-(0.25) * 3/32});
    \draw[thick,domain=(1/8):0.1875] plot (\x, {8*(\x-(1/8))^2-(0.25)*(3/32)});
    \draw[thick,domain=0.1875:0.499] plot (\x, {(-5/16)*exp((5/16)/(\x-.5) + 1) + (11/32) - (0.25)*(3/32)});
    \draw[thick,domain=0.499:0.5001] plot (\x, {(11/32) - (0.25) * (3/32)});
    \draw[thick,domain=0.5001:0.8125] plot (\x, {(5/16)*exp(-(5/16)/(\x-.5) + 1) + (11/32) - (0.25) * (3/32)});
    \draw[thick,domain=0.8125:0.875] plot (\x, {-8*(\x-(7/8))^2 + (22/32) - (0.25)*(3/32)});
    \draw[thick,domain=0.875:(14/16),color=red] plot (\x, {(22/32)-(0.25)*(3/32)});
    \draw[thick,domain=(14/16):(15/16),color=purple] plot (\x, {-8*(0.75)*(\x-(14/16))^2+(22/32) - (0.25)*(3/32)});
    \draw[thick,domain=(15/16):1,color=blue] plot (\x, {-0.75 * \x + ((32*1-3)*(0.75)/(32)) + (22/32) - (0.25)*(3/32)});
\end{tikzpicture}
    \caption{
        Examples of the component functions specified in  
        \eqref{eq:component-functions} with differing parameters of $d$ and $\delta$. 
        The parts that are highlighted correspond to 
        the first (blue), second (purple), third (red), 
        third to last (red), second to last (purple), and last (blue) parts of 
        \eqref{eq:component-functions}.
        The colored components are the sections of  
        \eqref{eq:component-functions}
        that were modified from previous works for
        Lemma \ref{prop-catastrophic-divergence}.
        Note when $d$ and $\delta$ are smaller than or equal to $1$, 
        the red components are gone;
        however, when either are greater than one the corresponding red component 
        is needed.
        \label{fig-component-function-examples}}
\end{figure}
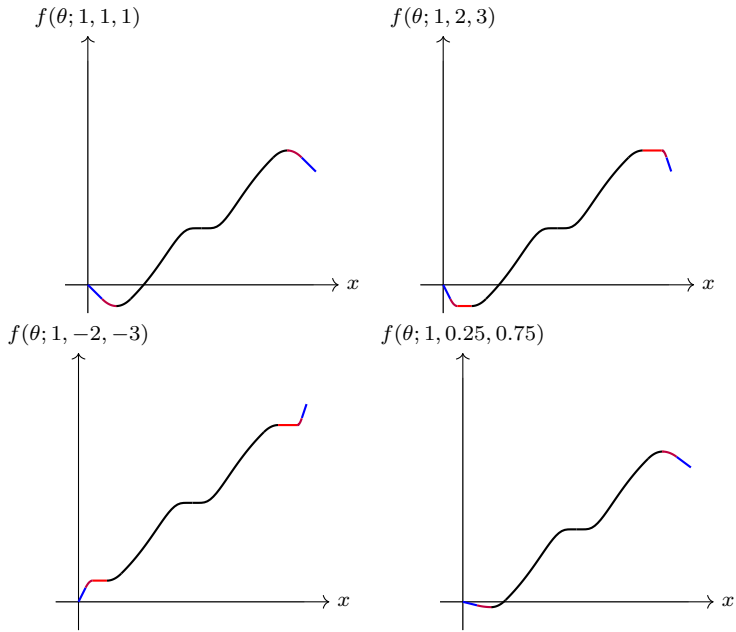

\begin{remark}
Compared to \cite{patel2024Gradient}, Frankenstein's function allows the derivative
at the two endpoints ($0$ and $m$) to be arbitrary real numbers in $\mathbb{R}$,
rather than being restricted to just $[0, 1]$.
As a result, we do not require further restrictions on our model of OFFO methods developed
in subsection \ref{subsec:offo-model}.
Additionally, it allows us to construct more interesting anti-convergence examples; 
in particular, we develop a new example (differing from our earlier technical report
\cite{varner2024Challenges}) 
for the algorithm of 
Malitsky and Mishchenko
\cite{malitsky2020Adaptive}.
\end{remark}

We now show that \eqref{eq:component-functions} satisfies Assumptions \ref{as:lower-bounded}
and \ref{as:locally-lipschitz-cont},%
\footnote{We consider differentiability at the end points to be one-sided derivatives.}
and its values and all derivatives at the end points are controlled. 
\begin{lemma}\label{lemma-component-function-properties}
    Let $m \in \mathbb{R}_{>0}$, $d \in \mathbb{R}$, and $\delta \in \mathbb{R}$.
    Let $a = \max(|d|, 1)$ and $b = \max(|\delta|, 1)$.
    Let $f(x; m, d, \delta) : (0, m] \to \mathbb{R}$ be defined in as in \eqref{eq:component-functions}.
    Then, the continuous extension of $f$ to $[0, m]$ is continuous;
    it is continuously differentiable on $(0, m)$; 
    its derivative is locally Lipschitz continuous on $(0,m)$;
    it admits one-sided derivatives of $-d$ and $-\delta$ at $0$ and $m$, respectively;
    it admits higher-order, one-sided derivatives of $0$ at $0$ and $m$; 
    $f(m; m, d, \delta) \geq m / 2$; and 
    $f(x; m, d, \delta) \geq -m / 8$.
\end{lemma}
\begin{proof} 
    The proof is analogous to that of Proposition 4.2 in \cite{patel2024Gradient}.
    \qed
    %See Lemma \ref{app-lemma-component-function-properties} in Appendix 
    %\ref{app:obj-construction}.
\end{proof}

We now show that their exists a function defined on $\mathbb{R}$ 
satisfying Assumptions \ref{as:lower-bounded} and \ref{as:locally-lipschitz-cont}, 
and, on an increasing sequence of points, $\{\phi_{k} : k + 1 \in \mathbb{N}\}$,
matches derivative values in a sequence $\{d_k : k + 1 \in \mathbb{N}\}$ by concatenation
of a series of Frankenstein's functions.
We provide an illustration of such an objective function in Figure \ref{fig-ce-example-main}.
\begin{lemma} \label{lemma-existence-of-objective}
    Let $\{\phi_j : j + 1 \in \mathbb{N}\} \subset \mathbb{R}$ be a
    strictly increasing unbounded sequence. 
    Let $\{d_j : j + 1 \in \mathbb{N}\} \subset \mathbb{R}$. 
    Then, there exists a locally Lipschitz smooth function,
    $F : \mathbb{R} \to \mathbb{R}$, depending on these sets such that
    $\dot F(\phi_j) = -d_j$,
    $F(\theta) \geq -1/8$, and 
    $F(\phi_j) \geq (\phi_j - \phi_0)/2$.
\end{lemma}
\begin{proof}
    The proof is analogous to that of Proposition 4.3 in \cite{patel2024Gradient}.
    \qed
    %See Lemma \ref{app-lemma:obj-exist} in Appendix \ref{app:obj-construction}. 
\end{proof}

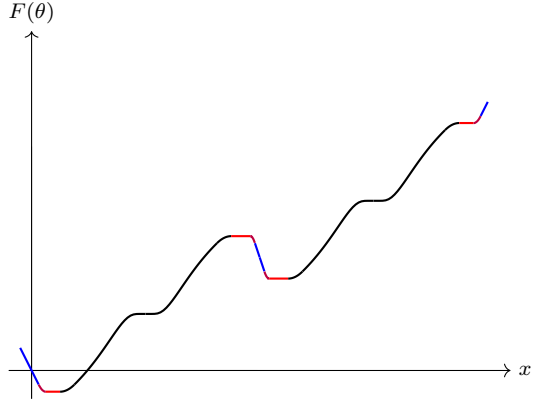
\begin{figure}
    \centering
    \begin{tikzpicture}[domain=0:1.5,scale=3]
    \draw[very thin,color=gray] (-0.1,-.125) grid (.99,.99);

    \draw[->] (-0.1, 0) -- (2.1, 0) node[right] {$x$};
    \draw[->] (0, -0.125) -- (0, 1.5) node[above] {$F(\theta)$}; 

    \draw[thick,domain=-0.05:0.03125,color=blue] plot (\x, -2 * \x);
    \draw[thick,domain=0.03125:0.0625,color=purple] plot (\x, {32*(\x-(1/16))^2 - (3/32)});
    \draw[thick,domain=0.0625:0.125,color=red] plot (\x, {-3/32});
    \draw[thick,domain=0.125:0.1875] plot (\x, {8*(\x-(1/8))^2-(3/32)});
    \draw[thick,domain=0.1875:0.499] plot (\x, {(-5/16)*exp((5/16)/(\x-.5) + 1) + (11/32) - (3/32)});
    \draw[thick,domain=0.499:0.5001] plot (\x, {(11/32) - (3/32)});
    \draw[thick,domain=0.5001:0.8125] plot (\x, {(5/16)*exp(-(5/16)/(\x-.5) + 1) + (11/32) - (3/32)});
    \draw[thick,domain=0.8125:0.875] plot (\x, {-8*(\x-(7/8))^2 + (22/32) - (3/32)});
    \draw[thick,domain=0.875:(46/48),color=red] plot (\x, {(22/32)-(3/32)});
    \draw[thick,domain=(46/48):(47/48),color=purple] plot (\x, {-8*3*3*(\x-(46/48))^2+(22/32) - (3/32)});
    \draw[thick,domain=(47/48):1,color=blue] plot (\x, {-3 * \x + ((32*3-3)*3/(32*3)) + (22/32) - (3/32)});
    
    % a = 3, c = -2
    \draw[thick,domain=1:{1+(1/48)},color=blue] plot (\x, {(-3)*(\x - 1) + .5});
    \draw[thick,domain={1+(1/48)}:{1+(1/24)},color=purple] plot (\x, {(8)*3*3*(\x-1-(1/(8*3)))^2 - (3/32)+ .5});
    \draw[thick,domain={1+(1/24)}:1.125,color=red] plot (\x, {-3/32+ .5});
    \draw[thick,domain=1.125:1.1875] plot (\x, {8*(\x-1-(1/8))^2-(3/32)+ .5});
    \draw[thick,domain=1.1875:1.499] plot (\x, {(-5/16)*exp((5/16)/(\x-1-.5) + 1) + (11/32) - (3/32)+ .5});
    \draw[thick,domain=1.499:1.5001] plot (\x, {(11/32) - (3/32)+ .5});
    \draw[thick,domain=1.5001:1.8125] plot (\x, {(5/16)*exp(-(5/16)/(\x-1-.5) + 1) + (11/32) - (3/32)+ .5});
    \draw[thick,domain=1.8125:1.875] plot (\x, {-8*(\x-1-(7/8))^2 + (22/32) - (3/32)+ .5});
    \draw[thick,domain=1.875:{1+((16*2-2)/(16*2))},color=red] plot (\x, {(22/32)-(3/32)+ .5});
    \draw[thick,domain={1+((16*2-2)/(16*2))}:{1+((16*2-1)/(16*2))},color=purple] plot 
    (\x, {-8*(-2)*2*(\x-1-(((16*2-2)/(16*2))))^2+(22/32) - (3/32)+ .5});
    \draw[thick,domain={1+((16*2-1)/(16*2))}:2,color=blue] plot 
    (\x, {(2)*(\x-1)+ ((32*2-3)*(-2)/(32*2)) + (22/32) - (3/32)+ .5});
\end{tikzpicture}
    \caption{
        The concatenation of two Frankenstein's functions.
        \label{fig-ce-example-main}}
\end{figure}

\subsection{Anti-Convergence Result for OFFO Methods}

We now state and prove our anti-convergence result.
\begin{proposition} \label{prop-catastrophic-divergence}
    Let $\theta_0 \in \mathbb{R}$, $\psi_0 \in \mathbb{R}$, 
    and $\mu_0 \in \mathbb{R}^m$. 
    Let $\mathcal{A}_{\offo}$ be a model for an OFFO method satisfying Conditions 
    \ref{condition-output-same} and \ref{condition-iterate-diverging}. 
    Then, there exists a function $F : \mathbb{R} \to \mathbb{R}$ satisfying
    Assumptions \ref{as:lower-bounded} and \ref{as:locally-lipschitz-cont} such that
    the sequence $\lbrace (\theta_k, \psi_k, \mu_k): k+1 \in \mathbb{N} \rbrace$, 
    defined by 
    \begin{equation}
        \theta_{k+1}, \psi_{k+1}, \mu_{k+1} = 
        \mathcal{A}_{\offo}(k, \theta_{k}, \psi_{k}, \mu_{k}, \dot F(\psi_{k})),
    \end{equation}
    satisfies $\liminf F(\theta_k) = \infty$ and $\liminf_k |d_k| \leq \liminf |\dot F(\theta_k)|$.
\end{proposition}
\begin{proof}
    By Condition \ref{condition-iterate-diverging}, there exists a set
    $\{d'_k : k + 1 \in \mathbb{N}\}$ and sequence $\{(\theta_k, \psi_k, \mu_k) : k + 1 \in \mathbb{N}\}$ 
    such that
    \begin{equation}
        \theta_{k+1}, \psi_{k+1}, \mu_{k+1} = 
        \mathcal{A}_{\offo}(k, \theta_{k}, \psi_{k}, \mu_{k}, d'_{k}),
        ~\forall (k+1) \in \mathbb{N},
    \end{equation}
    where $\{\theta_k : k + 1 \in \mathbb{N}\}$ and $\{\psi_k : k + 1 \in \mathbb{N}\}$ are strictly increasing and diverging.
    We define $F(\theta)$ by Lemma \ref{lemma-existence-of-objective}, which requires two sets:
    $\{\phi_j : j + 1 \in \mathbb{N}\}$ and $\{d_j : j + 1 \in \mathbb{N}\}$, which we construct as follows. 

    Let $\phi_0 = \min \lbrace \min_k \lbrace \theta_k \rbrace, \min_{k} \lbrace \psi_k \rbrace \rbrace$. Define $\lbrace \phi_j : j \in \mathbb{N} \rbrace$ recursively by 
    \begin{equation}
        \phi_{j} = \min \lbrace \min_{k : \theta_k > \phi_{j-1}} \lbrace \theta_k \rbrace, 
        \min_{k : \psi_k > \phi_{j-1}} \lbrace \psi_k \rbrace \rbrace. 
    \end{equation}
    In other words, $\lbrace \phi_k : k+1 \in \mathbb{N} \rbrace$ is a strictly increasing
    sequence made up by ordering $\lbrace \theta_k : k+1 \in \mathbb{N} \rbrace$ and
    $\lbrace \psi_k :k+1 \in \mathbb{N}\rbrace$. 
    Furthermore, since $\lbrace \theta_k: k+1 \in \mathbb{N} \rbrace$ is a 
    diverging (by Condition \ref{condition-iterate-diverging}) subsequence of 
    $\lbrace \phi_k : k+1 \in \mathbb{N} \rbrace$, then $\lbrace \phi_k : k+1 
    \in \mathbb{N} \rbrace$ diverges. 

    Next, $\forall (j+1) \in \mathbb{N}$, let
    \begin{equation}
        d_j = \begin{cases}
            1 & \phi_j \not\in \lbrace \psi_k : k+1 \in \mathbb{N} \rbrace \\ 
            -d_{\ell}' & \text{where}~ \phi_j = \psi_{\ell}.
        \end{cases}
    \end{equation}

    With $\{\phi_k : k + 1 \in \mathbb{N}\}$ and $\{d_k : k + 1 \in \mathbb{N}\}$,
   let $F(\theta)$ be given by Lemma \ref{lemma-existence-of-objective}.
    We now show that $\forall k \in \mathbb{N}$, $\theta_k, \psi_k, \mu_k = \mathcal{A}_{\offo}(k-1, \theta_{k-1}, \psi_{k-1}, \mu_{k-1}, \dot F(\psi_{k-1}))$
    by induction.
    As verifying the base case is analogous to the generalization step, we show only the 
    generlization step. 
    For the inductive hypothesis, suppose this is true for all $j \in \{1,...,k\}$ for $k \in \mathbb{N}$.

    By our construction, there exists an $\ell_{k} + 1 \in \mathbb{N}$ such that $\phi_{\ell_{k}} = \psi_{k}$ and $d_{\ell_{k}} = -d'_{k}$.
    Therefore, $\dot F(\psi_{k}) = \dot F(\phi_{\ell_{k}}) = -d_{\ell_{k}} = d'_{k}$. 
    Hence, $\theta_{k+1}, \psi_{k+1}, \mu_{k+1}
    = \mathcal{A}_{\offo}(k, \theta_{k}, \psi_{k}, \mu_{k}, \dot F(\psi_{k}))$.

    We now turn to the conclusion of the statement. Just as above, there exists an 
    $\ell_{k} + 1 \in \mathbb{N}$ such that $\phi_{\ell_k} = \theta_k$. 
    By Lemma \ref{lemma-existence-of-objective}, 
    $F(\theta_k) = F(\phi_{\ell_k}) \geq (\phi_{\ell_k} - \phi_0) = (\theta_k - \phi_0)/2$. 
    Since $\theta_k \to \infty$ as $k \to \infty$ by 
    Condition \ref{condition-iterate-diverging}, $\lim_{k} F(\theta_k) = \infty$. 
    Finally, $\liminf |d_k| \leq \liminf |\dot F(\theta_k)|$ since 
    $\{\dot F(\theta_k) : k + 1 \in \mathbb{N}\} \subseteq \{d_k : k + 1 \in \mathbb{N}\}$.
    \qed
\end{proof}

\section{Algorithm Design}      
\label{sec:algorithm}
For our computational pattern (expensive objective evaluations, and relatively 
inexpensive gradient evaluations) and function class (see Assumptions 
\ref{as:lower-bounded} and \ref{as:locally-lipschitz-cont}) of interest,
we recall that methods that check the objective multiple times before accepting a 
new iterate are computationally undesirable, while methods that completely avoid 
objective function evaluations seem desirable, 
yet can experience anti-convergence (see Section \ref{sec:counter}).
To address this gap, we introduce Algorithm \ref{alg:novel-step-size}: 
a novel method with economical, event-driven objective function evaluations 
that safeguard against divergence of the optimality gap.
Our algorithm has three critical, novel  components:
(1) events that trigger objective function evaluations;
(2) a non-sequential Armijo condition;
(3) and a step size strategy designed to greedily improve the optimality gap prior 
to a triggering event. 

\begin{algorithm}[!ht]
    \caption{Novel Gradient Descent Algorithm with Non-sequential Armijo}
    \label{alg:novel-step-size}
    \begin{algorithmic}[1]
    \Require $F, \dot F, \theta_{0}, \epsilon > 0, \rho \in (0, 1), 
    \overline{\delta} \in (0, \infty), \delta_0 \in (0, \overline{\delta}]$    
     
    % Outer Loop Counter
    \State $k \leftarrow 0$ \Comment{Outer loop counter}
   
	% Step Size 
    % \State $\delta_k \leftarrow 1$ \Comment{Step size scaling}

	% threshold for gradient
    \State $\tau_{\gra,\low}^0, \tau_{\gra,\upp}^0 \leftarrow \inlinenorm{\dot{F}(\theta_0)}_2/\sqrt{2}, \sqrt{10}\inlinenorm{\dot{F}(\theta_0)}_2$ \Comment{Test thresholds on gradient}
    
    % Outer Loop
    \While{$\inlinenorm{ \dot F(\theta_k) }_2 > \epsilon$} \Comment{Outer loop}
    
    % Inner Loop Counter
    \State $j, \psi_0^k \leftarrow 0, \theta_k$ \Comment{Inner loop counter and initialization}
    
    % Lipschitz approximation
    \State $\hat{L}_j^k \leftarrow \Call{Update}{j, k}$      
    \Comment{Initial local Lipschitz constant approximation}\label{line:alg-ns-lip-approx}
    
   	% Inner Loop
    \While{$\true$} \Comment{Inner Loop}
    		%Step Direction and Size
			\State $C_{j,1}^k \leftarrow 
				\inlinenorm{\dot{F}(\psi_j^k)}_2^3 + .5\inlinenorm{\dot{F}(\psi_j^k)}_2^2 \hat{L}_j^k + 10^{-16 }$
			\State $C_{j,2}^k \leftarrow \inlinenorm{\dot{F}(\psi_j^k)}_2 + .5 \hat{L}_j^k + 10^{-16}$
    		\State $\alpha^k_{j} \leftarrow 
            \min( [(\tau^{k}_{\gra,\low})^2/C_{j, 1}^k], 1/C_{j,2}^k ) + 10^{-16}$ 	\label{line:alg-ns-stepsize-calc}
    		\Comment{Adaptive step size}
    		%Check if in the exit or growth in gradient
			\State Call Algorithm \ref{alg:triggering-event-logic} to check for inner loop termination.	\label{line:call-trigger-check}
    		
    		\State $\psi_{j+1}^k, j \leftarrow \psi_j^k - \delta_k \alpha_j^k \dot{F}(\psi^k_j), j+1$  
			\Comment{Gradient step} 
            \State $\hat{L}_j^k \leftarrow \Call{Update}{j, k}$ \Comment{Update local Lipschitz constant approximation}
    \EndWhile %End inner loop

    \EndWhile %End Outer Loop
	\State \Return $\theta_k$    
    \end{algorithmic}
\end{algorithm}
\begin{remark}
	Parameters in algorithm \ref{alg:novel-step-size} and \ref{alg:triggering-event-logic}
    such as $\tau_{\gra,\low}^k, \tau_{\gra,\upp}^k, \delta_{k+1}$,
	and the max iteration for each inner loop can be selected from a much wider range.
    The values suggested here seem to work quite well on suites of test problems 
    (these results are not reported here).
\end{remark} 

\subsection{Event-based Objective Function Evaluations}
Consider an iterative optimization method initialized at some outer 
loop iterate, $\theta_0$, and envision an imaginary ball around this point. 
The iterative method may produce an iterate that escapes this imaginary ball, 
or it may remain entirely within the ball. 
If an iterate escapes the ball, this will trigger an objective function evaluation 
and possible changes to algorithm parameters. 
If the iterates remain in the ball, then the iterates are either 
\begin{enumerate}
\item Progressing towards a stationary point as measured by an iterate's gradient-norm
    dropping below a threshold relative to $\inlinenorm{\dot F(\theta_0)}_2$; 
\item Or enter into ``cycling''-type behavior in which the iterates remain in the ball 
    but do not make progress towards a stationary point.
\end{enumerate}
In the first case, dropping below a gradient-norm threshold relative to 
$\inlinenorm{\dot F(\theta_0)}_2$ would trigger an objective evaluation and 
possible algorithm parameter changes. 
In the second case, if the objective function has not been evaluated for some 
fixed number of iterates, then the objective function will be evaluated and 
algorithm parameters may be changed.

These triggering events are checked in Line \ref{line:call-trigger-check} 
of Algorithm \ref{alg:novel-step-size}.
Specifics of the triggering events are shown in Algorithm \ref{alg:triggering-event-logic},
which includes an optional but useful safeguard for checking if the gradient-norm 
of the iterates exceeds an upper threshold relative to $\inlinenorm{\dot F(\theta_0)}_2$. 

When a triggering event occurs, the most recent iterate is used as either the 
new outer loop iterate or the process is restarted at the most recent outer 
loop iterate with algorithm parameter adjustments that ensure descent 
(see Section \ref{sec:theory}).
We now discuss how the decision is made to accept the inner loop iterate most recent 
to a triggering event, or whether to discard it. 

\begin{algorithm}[!ht]
    \caption{Triggering Events and Post-Descent-Check Logic}
    \label{alg:triggering-event-logic}
    \begin{algorithmic}[1]
    \Require $F, \dot F, \theta_k, \psi_j^k, \delta_k, \alpha_0^k, \rho, \tau_{\gra,\low}^k, \tau_{\gra,\upp}^k, j, k$
    \If{ $\inlinenorm{ \psi_j^k - \theta_k }_2 > 10$ 
        \textbf{or} $\inlinenorm{ \dot F(\psi_j^k)}_2 \not\in (\tau_{\gra,\low}^k, \tau^k_{\gra,\upp})$ 
        \textbf{or} $j == 100$} 																								\label{line:alg-ns-te}
        \If{ $F(\psi_j^k) \geq F(\theta_k) - \rho \delta_k \alpha_0^k \inlinenorm{\dot{F}(\theta_k)}_2^2 $ }
            \State $\theta_{k+1}, \delta_{k+1} \leftarrow \theta_k, .5 \delta_k$												\label{line:alg-ns-decrease-sf}
            \Comment{Reset iterate with reduced step size scaling}
        \ElsIf{ $\inlinenorm{ \dot F(\psi^k_j)}_2 \leq \tau_{\gra,\low}^k$ }
            \State $\theta_{k+1}, \delta_{k+1} \leftarrow \psi_j^k, \delta_k$													\label{line:alg-ns-accept-lwgd}
        \Comment{Accept iterate and leave step size unchanged}    			
            \State $\tau_{\gra,\low}^{k+1}, \tau_{\gra,\upp}^{k+1} \leftarrow \inlinenorm{\dot{F}(\theta_{k+1})}_2/\sqrt{2}, \sqrt{20}\tau_{\gra,\low}^{k+1}$
        \ElsIf{$\inlinenorm{ \dot F(\psi^k_j)}_2 \geq \tau_{\gra,\upp}^k$}
            \State $\theta_{k+1}, \delta_{k+1}, \leftarrow \psi_j^k, \min\{1.5\delta_k, \overline{\delta}\}$									\label{line:alg-ns-accept-upgb}
            \Comment{Accept iterate and increase step size}
            \State $\tau_{\gra,\low}^{k+1}, \tau_{\gra,\upp}^{k+1} \leftarrow\inlinenorm{\dot{F}(\theta_{k+1})}_2/\sqrt{2}, \sqrt{20}\tau_{\gra,\low}^{k+1}$
        \Else
            \State $\theta_{k+1}, \delta_{k+1}, \leftarrow \psi_j^k, \min\{1.5\delta_k, \overline{\delta}\}$									\label{line:alg-ns-accept-ingb}
            \Comment{Accept iterate and increase step size}
            \State $\tau_{\gra,\low}^{k+1}, \tau_{\gra,\upp}^{k+1} \leftarrow \tau_{\gra,\low}^{k} \tau_{\gra,\upp}^{k} $
        \EndIf
        \State $k \leftarrow k+1$
        \State Exit Inner Loop
    \EndIf
    \end{algorithmic}
\end{algorithm}

\begin{remark}
    The use of a ball of size 10 in Line 1 in Algorithm \ref{alg:triggering-event-logic}
    can be selected from a wider range and does not need to be fixed. The value of 10 
    suggested here is based on an initial attempt at choosing this parameter that worked 
    quite well on problems, so we did not try to tune it further. There are reasons to
    make this parameter adaptive as discussed at the end of 
    Sections \ref{sec:analysis-grad-subsequence} and \ref{sec:convergence-rate},
    which we leave to future work.
\end{remark}

\subsection{Non-sequential Armijo Condition}

To determine whether to accept the terminal inner loop iterate as the next 
outer loop iterate, we use a novel evolution of Armijo's line search condition 
\cite{armijo1966Minimization}.
Recall, the standard Armijo condition is inspired by a first-order Taylor expansion model:
for the acceptance of a proposed iterate, $\psi$, 
$F(\psi)$ must be no greater than $ F(\theta) + \rho \dot F(\theta)^\intercal(\psi - \theta)$, 
where $\rho > 0$ is (typically) a small relaxation parameter and $\theta$ is 
some initial point \cite[Eq. 1.11]{bertsekas1999}. 
Our condition on the other-hand does not require such a local model:
we accepted a proposed iterate, $\psi$, if 
$F(\psi)$ is less than $F(\theta) - \rho \delta \alpha \inlinenorm{\dot F(\theta)}_2^2$,
where $\psi$ is not necessarily (and usually is not)  $\theta - \delta \alpha \dot F(\theta)$.%
\footnote{Note, $\psi = \theta - \delta \alpha \dot F(\theta)$, then the two conditions 
coincide. However, $\psi$ can only be equal to this if an event is triggered during the 
first inner loop iteration.}

\subsection{A Greedy Step Size Strategy}
\begin{algorithm}[!ht]
    \caption{Update Scheme for Local Lipschitz Approximation}
    \label{alg:update}
    \begin{algorithmic}[1]
        \Require $\dot{F}, \psi_{j}^k, \psi_{j-1}^k$
        \Function{Update}{$j,k$} \Comment{Method for Local Lipschitz Approximation}
            \If{$j == 0$ \textbf{and} $k == 0$}
                \State $\hat{L}_0^0 \leftarrow 1$
                \Comment{Initial value}
            \ElsIf{$j == 0$ \textbf{and} $k > 0$}
                \State $\hat{L}_0^k \leftarrow \hat{L}_{j_{k-1}}^{k-1}$
                \Comment{Use previous estimate}
            \ElsIf{$j > 0$ \textbf{and} $k \ge 0$ \textbf{and} $\theta_{k} \not= \theta_{k-1}$}
                \State $\hat{L}_{j}^k \leftarrow \frac{\inlinenorm{\dot{F}(\psi_j^k) - \dot{F}(\psi_{j-1}^k)}_2}{\inlinenorm{\psi_j^k-\psi_{j-1}^k}_2} $
                \Comment{Aggressive estimate}
            \ElsIf{$j > 0$ \textbf{and} $k \geq 0$ \textbf{and} $\theta_k = \theta_{k-1}$}
                \State $\hat{L}_{j}^k \leftarrow \max\bigpar{ \frac{\inlinenorm{\dot{F}(\psi_j^k) - \dot{F}(\psi_{j-1}^k)}_2}{\inlinenorm{\psi_j^k-\psi_{j-1}^k}_2}, \hat{L}_{j-1}^k }$
                \Comment{Conservative estimate}
            \EndIf
            \State \Return $\hat{L}_j^k$
        \EndFunction
    \end{algorithmic}
\end{algorithm} 

Finally, between objective function evaluations (i.e., between two triggering events),
we run an adaptive step size gradient descent routine so as to
to greedily improve the optimality gap without checking the objective function
(see lines \ref{line:alg-ns-lip-approx} and \ref{line:alg-ns-stepsize-calc} in 
Algorithm \ref{alg:novel-step-size}).
Like other works in adaptive step size algorithms, our step size is the minimum 
between two terms \cite{malitsky2023adaptive,zhang2019Why}, but takes advantage 
of the structure imposed by the triggering events: 
the first term is derived through a Zoutendjik's analysis combining information 
from our first and second triggering event; and
the second term is used as a safeguard against misbehavior of the first.%
\footnote{The appearance of the numerical constant, $10^{-16}$, is for numerical stability.}
Analogous to other step size strategies \cite{nesterov2012GradientMF,curtis2018exploiting,zhang2020firstorder},
each term attempts to approximate the local Lipschitz constant 
(see Algorithm \ref{alg:update}).
We emphasize that, while our step size is closely related to these other algorithms, 
it differs in that no local model is used to verify the accuracy of our estimate, 
and it only relies on local Lipschitz smoothness of the objective in its derivation.

\section{Global Convergence Analysis}                
\label{sec:theory}
Here, we conduct a (first-order) global convergence analysis of Algorithm \ref{alg:novel-step-size}
in two steps.
\begin{enumerate}
\item \emph{Accepting new iterates.} We show that eventually an outer loop must 
produce an iterate that satisfies our descent condition (i.e., accepting a new iterate) 
unless the algorithm already found a first-order stationary point.
\item \emph{Analysis of a gradient subsequence.} We analyze the asymptotic 
behavior of the gradients at the iterates to conclude that either a stationary 
point is found in finite time, the full sequence of gradients tend to zero, 
or a subsequence tends to zero under certain growth conditions on the local Lipschitz constant.
\end{enumerate}

\subsection{Accepting new iterates}

To begin, we introduce some notation.
Let $\theta_k$ be the $k$th outer loop iterate (starting with $k=0$); 
let $\{\psi_0^k,...,\psi_{j_k}^k\}$ be the inner loop iterates generated at outer 
loop iteration $k$ ($\psi_0^k = \theta_k$); and 
let the terminal inner loop iteration denoted by $j_k \in \mathbb{N},~j_k \in (0,100]$ 
(see Algorithm \ref{alg:novel-step-size}).\footnote{In other words, at least
one of the triggering events
evaluated to true in Line \ref{line:alg-ns-te} of algorithm \ref{alg:triggering-event-logic} 
at $\psi_{j_k}^k$, and each evaluated to false at $\psi_{j_k-1}^k$.}

When an event is triggered at iteration $k$ (see Algorithm \ref{alg:triggering-event-logic}),
either the proposed iterate satisfies the descent condition and is accepted---%
that is, $\theta_{k+1} \leftarrow \psi_{j_k}^k$---,
or the proposed iterate is rejected---$\theta_{k+1} \leftarrow \theta_k$.
To keep track of the accepted iterates, which are those that are distinct and 
satisfy our descent condition, let
\begin{equation} \label{eqn-iterate-subsequence}
\ell_0 = 0\quad\mathrm{and}\quad \ell_t = \min_k \lbrace k > \ell_{t-1}: 
  \theta_k \neq \theta_{\ell_{t-1}} \rbrace, ~ \forall t \in \mathbb{N},
\end{equation}
with the convention that $\ell_t = \infty$ if no finite $k$ can be found to 
satisfy the property; and $\ell_{t} = \infty$ if $\ell_{t-1} = \infty$. 

Our analysis of accepting new iterates proceeds in four steps.
\begin{enumerate}
  \item We show that our step sizes, $\alpha_j^k$ are bounded away from zero and 
  from above. See Lemma \ref{result-bounded-step-size}.
  \item We show that all inner loop iterates remain in a compact set; and 
  if an inner loop iterate is not accepted, then the next set of inner loop iterates 
  remains in a compact set contained in the previous one. We use this result 
  to control the local Lipschitz constant. See Lemma \ref{result-compact-sets}
  \item Using the above two results, we show that if the adjustment parameter, 
  $\delta_k$, is sufficiently small, then a proposed iterate must be accepted. 
  See Lemma \ref{result-sufficient-scaling}.
  \item Finally, we show that this adjustment parameter will eventually occur, 
  which ensures that a proposed iterate will be accepted. 
  See Theorem \ref{result-accept-new-iterates}.
\end{enumerate}

\begin{lemma} \label{result-bounded-step-size}
Let $\underline \alpha = 10^{-16}$ and $\overline\alpha = 10^{-16} + 10^{16}$.
Then, for all $k+1 \in \mathbb{N}$ and for all $j+1 \in \mathbb{N}$ the step size $\alpha_j^k \in [\underline{\alpha}, \overline \alpha]$.
\end{lemma}
\begin{proof}
From line \ref{line:alg-ns-stepsize-calc} in algorithm \ref{alg:novel-step-size}, $\alpha_j^k \geq 10^{-16}$, and
\begin{equation}
    \alpha_j^k \leq \frac{1}{\inlinenorm{\dot F(\psi_j^k)}_2 + .5\hat{L}_j^k + 10^{-16}} + 10^{-16} \leq 10^{16} + 10^{-16},
\end{equation}
since $\hat L_{j}^k$ is always non-negative (see Algorithm \ref{alg:update}).
\qed
\end{proof}

\begin{lemma} \label{result-compact-sets}
Suppose \eqref{eqn-problem} satisfies assumption \ref{as:locally-lipschitz-cont}, 
and is solved using algorithm \ref{alg:novel-step-size} 
initialized at any $\theta_0 \in \mathbb{R}^n$. 
Then, for every $k+1 \in \mathbb{N}$, there exists a compact $C_k \subset \mathbb{R}^n$ such that $\theta_k \in C_k$, $\lbrace \psi_{1}^k,\ldots, \psi_{j_k}^k \rbrace \subset C_k$, and, if $\theta_{k} = \theta_{k+1}$, then $C_{k} \supset C_{k+1}$.
\end{lemma}
\begin{proof}
We use Lemma \ref{result-bounded-step-size} to construct the desired compact sets.
For any $\theta \in \mathbb{R}^n$, define $\mathcal{B}(\theta) = \lbrace \psi: 
\inlinenorm{ \psi - \theta}_2 \leq 10 \rbrace$ (note, $10$ comes from the 
radius bound in Algorithm \ref{alg:triggering-event-logic}).
Define $\mathcal G(\theta) = \sup_{\psi \in \mathcal{B}(\theta)} 
\inlinenorm{ \dot F(\psi) }_2$, which is finite because of the continuity of the 
gradient (see Assumption \ref{as:locally-lipschitz-cont}) and compactness of 
$\mathcal{B}(\theta)$. 
Define
\begin{equation} \label{eqn-compactset-outerloop}
C_k = \lbrace \psi : \inlinenorm{ \psi - \theta_k }_2 \leq 10 + \delta_k 
\overline\alpha \mathcal G(\theta_k) \rbrace, ~\forall k+1 \in \mathbb{N}.
\end{equation}
We first show $C_k$ satisfies the desired properties. First, $\theta_k \in C_k$ as the radius of the ball defining $C_k$ is non-negative. Second, by definition, the inner loop iterate $\psi_{i}^k$ satisfies $\inlinenorm{ \psi_{i}^k - \theta_k}_2 \leq 10$ for $i=1,\ldots,j_k-1$. 
In other words, $\psi_{i}^k \in \mathcal{B}(\theta_k)$ for $i=1,\ldots,j_k-1$. 

By the upper bounds on the step size and by assumption \ref{as:locally-lipschitz-cont},
\begin{align}
\inlinenorm{ \psi_{j_k}^k - \theta_k} _2 
&\leq \inlinenorm{ \psi_{j_k}^k - \psi_{j_k-1}^k}_2 + \inlinenorm{\psi_{j_k -1}^k - \theta_k}_2 \\
&\leq \inlinenorm{ \delta_k \alpha_{j_k-1}^k \dot F(\psi_{j_k -1 }^k) }_2 + 10 \\
&\leq \delta_k \overline \alpha \mathcal G(\theta_k) + 10.
\end{align}
To summarize, $\psi_{j_k}^k \in C_k$ and $\psi_{i}^k \in \mathcal{B}(\theta_k) \subset C_k$ for $i=1,\ldots,j_k-1$. 

Finally, when $\theta_{k} = \theta_{k+1}$, $\delta_{k+1} = .5 \delta_k < \delta_k$. 
Plugging this in (\ref{eqn-compactset-outerloop}), $C_{k} \supset C_{k+1}$.
\qed
\end{proof}

\begin{lemma} \label{result-sufficient-scaling}
Suppose \eqref{eqn-problem} satisfies assumption \ref{as:locally-lipschitz-cont}, 
and is solved using algorithm \ref{alg:novel-step-size} 
initialized at any $\theta_0 \in \mathbb{R}^n$.
Let $C \subset \mathbb{R}^n$ be a compact set such that $C_k \subseteq C$ 
for some $k+1 \in \mathbb{N}$, where $\lbrace C_k : k+1 \in \mathbb{N} \rbrace$ 
are defined in Lemma \ref{result-compact-sets}.
Let $\mathcal L(C)$ denote the Lipschitz rank of the gradient function over $C$
(which we can take to be positive). 
If $\dot F(\theta_k) \neq 0$ and
\begin{equation}
\delta_k < \frac{2 (1 - \rho) }{ \mathcal L(C) \overline \alpha},
\end{equation}
then $F(\psi_{j_k}^k) < F(\theta_k) - \rho \delta_k \alpha_{0}^k \inlinenorm{\dot F(\theta_k)}_2^2$. 
\end{lemma}

\begin{proof}
We recall several facts. 
By Lemmas \ref{result-bounded-step-size} and \ref{result-compact-sets}, $\lbrace \psi_{j}^k : j=0,\ldots, j_k \rbrace \subset C_k \subseteq C$, and $0 < \underline\alpha \leq \alpha_j^k \leq \overline\alpha < \infty$ for $j = 0,\ldots j_k-1$. 
Moreover, $j_k \neq 0$.

Now, we convert the hypothesis on $\delta_k$ into a more useful form. 
Specifically, by Lemma \ref{result-bounded-step-size}, $\delta_k \alpha_j^k \leq \delta_k \overline\alpha < [2(1-\rho) ]/\mathcal L(C)$ for $j=0,\ldots,j_k -1$. 
By algorithm \ref{alg:novel-step-size} and Lemma \ref{result-bounded-step-size}, $\delta_k \alpha_j^k > 0$, which implies
\begin{equation} \label{eq-for-taylors-theorem}
  \frac{(\delta_k \alpha_j^k)^2 \mathcal L(C)}{2} < (\delta_k \alpha_j^k)(1-\rho).
\end{equation}

By Taylor's theorem, Assumption \ref{as:locally-lipschitz-cont} and (\ref{eq-for-taylors-theorem}), for $j=0,\ldots,j_k-1$,   
\begin{align}
F(\psi_{j+1}^k) &\leq F(\psi_j^k) - \alpha_j^k \delta_k \inlinenorm{\dot F(\psi_j^k)}_2^2 + \frac{\mathcal L(C)}{2} \inlinenorm{ \delta_k \alpha_j^k \dot F(\psi_j^k) }_2^2\\
&\leq F(\psi_j^k) - \alpha_j^k \delta_k \inlinenorm{\dot F(\psi_j^k)}_2^2 + (\delta_k \alpha_j^k)(1-\rho) \inlinenorm{\dot F(\psi_j^k) }_2^2 \\
&\leq F(\psi_j^k) - \rho\alpha_j^k \delta_k \inlinenorm{\dot F(\psi_j^k)}_2^2 \label{eq-telescope-relation}\\
&\leq F(\psi_0^k) - \rho\delta_k\sum_{i=0}^j \alpha_i^k \inlinenorm{\dot F(\psi_i^k)}_2^2.
\end{align}
Since $\inlinenorm{\dot F(\psi_0^k)}_2 = \inlinenorm{\dot F(\theta_k)} \neq 0$
by hypothesis  
and $\alpha_i^k \geq \underline \alpha > 0$ 
by Lemma \ref{result-bounded-step-size}, the last inequality must be strict.
\qed
\end{proof}

\begin{theorem} \label{result-accept-new-iterates}
Suppose \eqref{eqn-problem} satisfies assumption \ref{as:locally-lipschitz-cont}, 
and is solved using algorithm \ref{alg:novel-step-size} initialized at any 
$\theta_0 \in \mathbb{R}^n$. 
Let $\lbrace \ell_t : t+1 \in \mathbb{N} \rbrace$ be defined as in 
\eqref{eqn-iterate-subsequence}. 
Then, for any $t+1\in \mathbb{N}$, if $\ell_t < \infty$ and $\dot F(\theta_{\ell_t}) \neq 0$, then $\ell_{t+1} < \infty$.
\end{theorem}
\begin{proof}
The proof is by induction. 
As the proof of the base case (i.e., $\ell_1 < \infty$) uses the same argument as the conclusion (i.e., if $\ell_t < \infty$ then $\ell_{t+1} < \infty$), we show the conclusion. 
To this end, suppose $\lbrace \ell_0,\ldots,\ell_{t} \rbrace$ are finite and $\dot F(\theta_{\ell_t}) \neq 0$. 
For a contradiction, suppose $\theta_k = \theta_{\ell_t}$ for all $k \geq \ell_t$. 
Then, by Lemma \ref{result-compact-sets}, the compact sets $C_{k} \subseteq C_{\ell_t}$ for all $k \geq \ell_t$. 
By Algorithm \ref{alg:triggering-event-logic}, $\delta_k = (.5)^{k-\ell_t} \delta_{\ell_t} > 0$,
which implies that there exists a $k \geq \ell_t$ such that
\begin{equation}
\delta_k = (.5)^{k-\ell_t} \delta_{\ell_t} < \frac{2 (1 - \rho) }{ \mathcal L(C_{\ell_t}) \overline \alpha}.
\end{equation}
By Lemma \ref{result-sufficient-scaling} and $\dot F(\theta_k) = \dot F(\theta_{\ell_t}) \neq 0$, 
$\psi_{j_k}^k$ satisfies our descent condition.
Hence, $\theta_{k+1} = \psi_{j_k}^k \neq \theta_{\ell_t}$. 
Therefore, $\ell_{t+1} = k+1 < \infty$.
\qed
\end{proof}

\begin{corollary} \label{coro-decrease-objective}
  Under the conditions of Theorem \ref{result-accept-new-iterates}, for any fixed
  $T \in \mathbb{N}$, $F(\theta_T) \leq F(\theta_0)$. Moreover, if $\dot F(\theta_{0}) \not = 0$,
  then there exists a $T^* \in \mathbb{N}$ such that for all $t \geq T^*$,
  $F(\theta_t) < F(\theta_0)$.
\end{corollary}

As corollary \ref{coro-decrease-objective} illustrates, 
theorem \ref{result-accept-new-iterates} ensures that our method cannot experience
catastrophic divergence on functions $F(\theta)$ satisfying assumption
\ref{as:locally-lipschitz-cont}; thereby, avoiding the pitfalls of OFFO methods 
for this function class (see Proposition \ref{prop-catastrophic-divergence}). 

\subsection{Analysis of a Gradient Subsequence} \label{sec:analysis-grad-subsequence}
We now study the accepted iterates, $\{\theta_{\ell_t} : t + 1 \in \mathbb{N}\}$, 
to show that the gradient function evaluated along this sequence 
(or, along a subsequence) is well-behaved relative to 
the local Lipschitz rank. 

\begin{theorem} \label{result-zoutendjik}
Suppose \eqref{eqn-problem} satisfies assumptions \ref{as:lower-bounded} and 
\ref{as:locally-lipschitz-cont}, and is solved using algorithm \ref{alg:novel-step-size}  
initialized at $\theta_0 \in \mathbb{R}^n$. 
Let $\lbrace \ell_t : t+1 \in \mathbb{N} \rbrace$ be defined as in \eqref{eqn-iterate-subsequence}.
Then, one of the following occurs.
\begin{enumerate}
\item There exists a $t + 1 \in \mathbb{N}$ such that $\ell_t < \infty$ and $F(\theta_{\ell_t}) \leq F(\theta_0)$ and $\dot F(\theta_{\ell_t}) = 0$.
\item The elements of $\lbrace \ell_t: t+1 \in \mathbb{N} \rbrace$ are all finite, $\liminf \delta_k > 0$, and
\begin{equation}
  \lim_{k \to \infty} \inlinenorm{\dot F(\theta_k)}_2 = 0
\end{equation}

\item Let $\lbrace C_k : k+1 \in \mathbb{N} \rbrace$ be a sequence of compact sets as in Lemma \ref{result-compact-sets}.
The elements of $\lbrace \ell_t : t+1 \in\mathbb{N} \rbrace$ are all finite, $\liminf \delta_k = 0$, and there exists a subsequence $\mathcal T \subseteq \mathbb{N}$ such that
\begin{equation} \label{eqn-relative-growth-gradient-lipschitz}
  \sum_{t\in\mathcal T}
  \frac{\inlinenorm{\dot F(\theta_{\ell_{t-1}})}_2^2}{\mathcal L(C_{\ell_{t-1}})} < \infty.%
  \footnote{Alternative, we can write $\sum_{t \in \mathcal{T}} \delta_{\ell_t -1}
  \inlinenorm{\dot F(\theta_{\ell_{t-1}})}_2^2 < \infty,$ which is similar to many 
  convergence results for stochastic gradient methods.}
\end{equation}
\end{enumerate}
\end{theorem}
\begin{proof}
We show that either $\{\ell_t : t + 1 \in \mathbb{N}\}$ are all finite, or a stationary point is found by induction. 
As the inductive step is the same as the base case, we only show the inductive step.
(Hypothesis:) Suppose that for $t + 1 \in \mathbb{N}$, that $\{\ell_0, ..., \ell_t\}$ are all finite. 
Then, either $\dot F(\theta_{\ell_t}) = 0$ or $\dot F(\theta_{\ell_t}) \not= 0$.
\begin{enumerate}
  \item If $\dot F(\theta_{\ell_t}) = 0$, $\theta_{\ell_t}$ is a stationary point, $F(\theta_{\ell_t}) \leq F(\theta_0)$ and $\ell_{t+1} = \infty$ by definition of $\ell_t$
(see (\ref{eqn-iterate-subsequence})).
  \item If $\dot F(\theta_{\ell_t}) \neq 0$, then $\ell_{t+1} < \infty$ by  Theorem \ref{result-accept-new-iterates}.
\end{enumerate}
Thus, by induction either Case 1 of Lemma \ref{result-zoutendjik} occurs or the 
elements of $\lbrace \ell_t : t+1 \in \mathbb{N} \rbrace$ are all finite. 

When the elements of $\lbrace \ell_t : t+1 \in \mathbb{N} \rbrace$ are all finite,
we consider two cases: $\liminf_k \delta_k > 0$ or $\liminf \delta_k = 0$.
\begin{enumerate}
\item If $\liminf_k \delta_k > 0$, then $\exists \underline \delta > 0$ such that
  $\delta_{k+1} > \underline \delta$ for all $k + 1 \in \mathbb{N}$.
  By our descent condition,
  \begin{equation} \label{eqn-guaranteed-reduction}
    F(\theta_{\ell_{t}}) < F(\theta_{\ell_{t}-1}) - \rho \delta_{\ell_{t}-1} \alpha_0^{\ell_{t}-1} \inlinenorm{ \dot F(\theta_{\ell_{t}-1})}_2^2, ~\forall t \in \mathbb{N}.
  \end{equation}
  Substituting the lower bound $\underline \delta$ and $\underline \alpha$ (see Lemma \ref{result-bounded-step-size}), using $\theta_{\ell_t-1} = \theta_{\ell_{t-1}}$, and re-arranging, we obtain
  \begin{equation}
    \frac{F(\theta_{\ell_{t-1}}) - F(\theta_{\ell_t})}{\rho \underline\delta \underline \alpha} > \inlinenorm{ \dot F(\theta_{\ell_{t-1}})}_2^2.
  \end{equation}
  Telescoping the sequence,
  \begin{equation} \label{eqn-zoutendjik-adjustment-parameter-lower-bounded}
    \sum_{t=1}^\infty \inlinenorm{\dot F(\theta_{\ell_{t-1}})}_2^2 < \frac{F(\theta_0) - F_{l.b.}}{\rho \underline\delta \underline{\alpha}} < \infty,
  \end{equation}
  which implies $\lim_{t\to\infty} \inlinenorm{\dot F(\theta_{\ell_t})}_2 = 0$. As
  $\theta_k = \theta_{\ell_t}, ~\forall k \in [\ell_t, \ell_{t + 1}) \cap \mathbb{N}$, 
  this implies that $\lim_{k\to\infty} \inlinenorm{\dot F(\theta_k)}_2 = 0$.
\item Suppose $\liminf_k \delta_k = 0$.
For a contradiction, suppose $\cup_{k=0}^\infty C_k$ is bounded. 
Then, there exists a compact set $C$ such that $\cup_{k=0}^\infty C_k \subset C$. 
Then, by Lemma \ref{result-sufficient-scaling}, 
\begin{equation}
  0 < \min\left(\frac{(1 - \rho)}{\mathcal{L}(C)\overline{\alpha}}, \delta_0\right) \leq \delta_k,
\end{equation}
which is a contradiction. Hence, $\cup_{k=0}^\infty C_k$ is unbounded.\medskip 

We now define $\mathcal T \subset \mathbb{N}$ 
to be the set of all $t\in \mathcal{T}$ such that $\ell_{t} > \ell_{t-1} + 1$ (i.e., there was a rejected outer loop iterate).
This set must be non-empty and infinite since $\liminf_{k} \delta_k = 0$.
Furthermore, $\forall t \in \mathcal{T}, ~\delta_{\ell_{t}-1} > (1-\rho)/[\overline \alpha \mathcal{L}(C_{\ell_{t-1}})]$ by Lemma \ref{result-sufficient-scaling}.
Therefore, considering our descent condition again, substituting the lower bound on 
$\delta_{\ell_{t}-1}$, and substituting the lower bound on the step size,
\begin{equation}
  F(\theta_{\ell_{t}}) < F(\theta_{\ell_{t}-1}) - \rho \frac{\underline\alpha (1-\rho)}{\overline \alpha \mathcal{L}(C_{\ell_{t-1}})} \inlinenorm{ \dot F(\theta_{\ell_{t}-1})}_2^2, ~\forall t \in \mathcal{T}.
\end{equation}
Multiplying by the constant terms, and telescoping, we obtain
\begin{equation}
  \sum_{t\in\mathcal T} \frac{\inlinenorm{\dot F(\theta_{\ell_{t-1}})}_2^2}{\mathcal L(C_{\ell_{t-1}})} < \infty.
\end{equation}
\end{enumerate}
This finishes the proof.
\qed
\end{proof}

To understand the above result, we make two comments. 
First, informally, if $\lbrace \theta_k : k + 1 \in \mathbb{N} \rbrace$ is bounded, 
then $\inf_k \delta_k > 0$,
which implies that $\lim_{k \to \infty} \inlinenorm{\dot F(\theta_k)} = 0$ by
part 2 of theorem \ref{result-zoutendjik}. 
Formally,
\begin{corollary} \label{result-adjustment-parameter-lower-bound}
Suppose \eqref{eqn-problem} satisfies assumption
\ref{as:locally-lipschitz-cont}, and is solved using algorithm \ref{alg:novel-step-size}  
initialized at $\theta_0 \in \mathbb{R}^n$. 
Let $\lbrace \ell_t : t+1 \in \mathbb{N} \rbrace$ be defined as in \eqref{eqn-iterate-subsequence}.
If $\lbrace \theta_k : k + 1 \in\mathbb{N} \rbrace$ is bounded, then $\inf_k \delta_k > 0$. 
\end{corollary}
\begin{proof}
  Let $\mathcal{B}$ denote the closure of a bounded set containing 
  $\lbrace \theta_k : k + 1 \in \mathbb{N} \rbrace \cup 
  \lbrace \psi : \exists k + 1 \in \mathbb{N}, ~\inlinenorm{\psi - \theta_k} \leq 10 \rbrace$. 
  Let $\mathcal{G} = \sup_{\psi \in \mathcal{B}} \inlinenorm{ \dot F(\psi) }_2$.
  Let $\lbrace C_k : k + 1 \in \mathbb{N} \rbrace$ be defined as in \eqref{eqn-compactset-outerloop}.
  Then, following the same reasoning as in the proof of Lemma \ref{result-compact-sets},
  for all $k + 1 \in \mathbb{N}$, $C_k \subseteq C$, where 
  \begin{equation}
    C := \lbrace \psi: \inf_{\phi \in \mathcal{B}} \norm{\psi - \phi}_2 \leq 
    10 + \overline\delta \overline\alpha \mathcal{G} \rbrace.
  \end{equation} 
  By Lemma \ref{result-sufficient-scaling}, for all $k + 1 \in \mathbb{N}$,
  \begin{equation}
    0 < \min\left( \frac{1-\rho}{\mathcal{L}(C)\overline{\alpha}}, \delta_0 \right) \leq \delta_k.
  \end{equation}
  \qed
\end{proof}

Second, \eqref{eqn-relative-growth-gradient-lipschitz} states that the norm of 
the gradient cannot outpace the local Lipschitz constant. 
Thus, for functions where this local constant can be controlled relative 
to the optimality gap or norm of the gradient squared, we can ensure 
that the iterates approach a region where the gradient is zero. 
This motivates developing control mechanisms for $\lbrace C_k \rbrace$ so as to
constrain the growth the local Lipschitz constants.
One possible avenue is by adaptively controlling the radius of the triggering
events,\footnote{A similar strategy for a specific class of linear regression problems
was considered in \cite{berahas2023nonuniformsmoothnessgradientdescent}.}
which would make the algorithm more reliable and potentially improve performance.
We leave such considerations to future work.

\section{Global and Local Convergence Rates}
\label{sec:convergence-rate}
We establish global and local rates of convergence for the gradient and objective 
when the iterates (eventually) remain bounded.
We can show that iterates remain bounded by considering one of two contexts.
\begin{enumerate}
    \item The first context is that the iterates enter a bounded level set of 
    the objective function---that is, for some $s \in \mathbb{R}$, there is an 
    iterate in the set $\lbrace \psi: F(\psi) \leq s \rbrace$, 
    where this set is bounded. 
    This is a textbook condition used for the global 
    convergence analysis of methods that make use of descent conditions
    \cite[Theorem 3.2]{nocedal2006Numerical}. 
    Here, we use it to make a stronger statement that algorithm will be stopped 
    at a near-stationary point in sub-linear time relative to the accepted outer loop 
    iterates (see Theorem \ref{thm-rate-on-gradient}).
    Our statement is more akin to more contemporary complexity results that 
    require the strong condition of global Lipschitz continuity (or its more 
    recent generalization) of the gradient function 
    (e.g., \cite[Ch. 2]{cartis2022Evaluationa}).
    \item The second context is that the iterates enter an ``isolated'' region 
    that satisfies a local Polyak-{\L}ojasiewicz (PL) condition. 
    The Polyak-{\L}ojasiewicz is typically assumed to hold globally in order to derive 
    global rate-of-convergence results (e.g., \cite{karimi2016linear,liu2022Loss}).
    In a recent excellent work, the PL condition is shown to hold locally for a broad 
    class of functions, which was further shown to ensure that the differential 
    equation formulation of gradient descent generated paths that would effectively 
    be trapped in this local region and converge to a stationary point 
    \cite{josz2023global}.
    We suggest an analogue of this condition that is appropriate for algorithms 
    that are discrete, and, in particular, appropriate for our method (%
    see Definition \ref{def-isolated-pl-region}). 
    We then use this condition to derive a local linear rate of convergence
    for our method (see Theorem \ref{thm-locally-pl-convergence-rate}).
\end{enumerate}

\begin{theorem} \label{thm-rate-on-gradient}
    Suppose \eqref{eqn-problem} satisfies assumptions \ref{as:lower-bounded} and 
    \ref{as:locally-lipschitz-cont}, and is solved using algorithm \ref{alg:novel-step-size}  
    initialized at $\theta_0 \in \mathbb{R}^n$. 
    Let $\lbrace \ell_t : t+1 \in \mathbb{N} \rbrace$ be defined as in \eqref{eqn-iterate-subsequence}.
    If there exists a $k + 1 \in \mathbb{N}$ such that $\theta_k$ is in a bounded 
    level set of the objective function, then either the algorithm terminates 
    in finite time at a stationary point or, $\forall T \in \mathbb{N}$,
    \begin{equation}
        \min_{t \in \lbrace 0,\ldots,T-1 \rbrace} \inlinenorm{\dot F(\theta_{\ell_t})}_2^2
        \leq \frac{F(\theta_0) - F_{l.b.}}{T\rho\underline{\delta}\underline{\alpha}}.
    \end{equation}
\end{theorem}
\begin{proof}
    By our non-sequential Armijo condition, if there exists a $k+1 \in \mathbb{N}$ 
    such that $\theta_k$ is in a bounded level set of the objective, then 
    $\theta_{k'}$ must remain in this bounded level set for all $k' \geq k$. 
    Hence, $\lbrace \theta_k: k+1 \in \mathbb{N} \rbrace$ is bounded.
    By Theorem \ref{result-zoutendjik} and Corollary \ref{result-adjustment-parameter-lower-bound},
    either the method terminates in finite time at a stationary point or
    $\exists \underline \delta > 0$ such that $\delta_k \geq \underline \delta$
    for all $k+1 \in \mathbb{N}$. 
    In the second case, the result follows by \eqref{eqn-zoutendjik-adjustment-parameter-lower-bounded}.
    \qed
\end{proof}

This classical rate can additionally be improved when a minimizing set satisfies 
what we call an isolated, local PL condition.
Informally, our isolated, local PL condition is akin to the notion of an 
isolated, local minimizer, as seen by the example in Figure \ref{figure-isolated-pl}.
As seen in Figure \ref{figure-isolated-pl}, an isolated, local PL region is a 
compact region in which the objective function satisfies the PL condition, and 
is surrounded by a sufficiently large neighborhood in which the objective function 
is strictly larger than the objective function in the compact PL region. We make this
notion formal in definition \ref{def-isolated-pl-region}. 

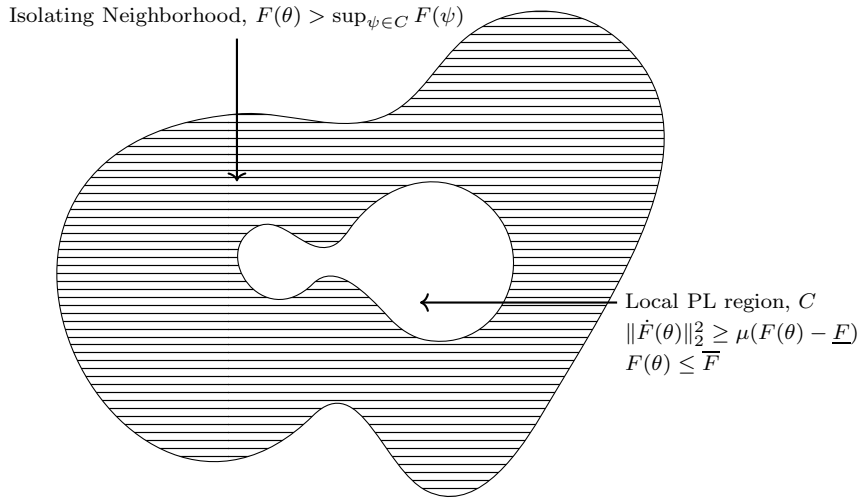
\begin{figure}[b!]
    \centering 
    \begin{tikzpicture}
    \coordinate (A) at (0,-2);
    \coordinate (B) at (1.5,-3);
    \coordinate (C) at (3,-2);
    \coordinate (D) at (4,3);
    \coordinate (E) at (2,3);
    \coordinate (F) at (1,2);
    \coordinate (G) at (-1,2);
    \draw[pattern=horizontal lines] (A) to [pattern=north east lines, closed, curve through = {(A) (B)  (C)  (D) (E)  (F) (G)}] (A);

    \coordinate (A1) at (0,-0.25);
    \coordinate (B1) at (1.0,-0.7);
    \coordinate (C1) at (1.5,-1);
    \coordinate (D1) at (2,1);
    \coordinate (E1) at (0.5, 0.5);
    \coordinate (F1) at (0.25,0.25);
    \coordinate (G1) at (-0.75,0.5);
    \draw[fill=white] (A1) to [closed, curve through ={(A1) (B1) (C1) (D1) (E1) (F1) (G1)}] (A1);

    \node (iso) at (-1,1) {};
    \node (isolab) at (-1,3) [anchor=south] {Isolating Neighborhood, $F(\theta) > \sup_{\psi \in C} F(\psi)$};
    \draw[->, thick] (isolab) to (iso); 

    \node (pl) at (1.3,-0.5) {};
    \node (pllab) at (4,-0.5) [anchor=west] {Local PL region, $C$};
    \node at (4,-0.9) [anchor=west] {$\inlinenorm{\dot F(\theta)}_2^2 \geq \mu(F(\theta) - \underline F)$};
    \node at (4,-1.3) [anchor=west] {$F(\theta) \leq \overline F$};
    \draw[->, thick] (pllab) to (pl);
\end{tikzpicture}
    \caption{A region satisfying  the isolated, local PL condition.
    \label{figure-isolated-pl}}
\end{figure}

\begin{definition} \label{def-isolated-pl-region}
    Let $F : \mathbb{R}^n \to \mathbb{R}$ be a differentiable function.
    Let $C \subset \mathbb{R}^n$ be compact. 
    Let $R > 0$.
    $C$ satisfies the isolated, local Polyak-{\L}ojasiewicz condition for 
    function $F$ of size $R$, if the following conditions are satisfied. 
    \begin{enumerate}
        \item $\exists \mu > 0$ for all $\theta \in C$ such that 
        \begin{equation}
            \norm{\dot F(\theta)}_2^2 \geq \mu (F(\theta) - \underline F),
        \end{equation}
        where $\underline F = \inf_{\psi \in C} F(\psi)$.
        \item For all $\theta \in \mathbb{R}^n$ in
        \begin{equation}
            \lbrace \theta \in \mathbb{R}^n : \mathrm{dist}(\theta, C) \leq R \rbrace 
            - C, 
        \end{equation}
        $F(\theta) > \sup_{\psi \in C} F(\psi)$, where $\mathrm{dist}(\theta,C)
        = \inf_{\psi \in C} \inlinenorm{\theta - \psi}_2$.
    \end{enumerate}
    If $C$ satisfies the isolated, local PL condition, then we refer to $C$ 
    as an isolated, local PL region. 
\end{definition}

Using this concept, we can prove a local rate-of-convergence result that is 
analogous to textbook local rate-of-convergence results for isolated minimizers
(e.g., see \cite[Theorems 3.5, 4.9, and 6.6]{nocedal2006Numerical}).
While there is fair criticism on the existence of an isolated, local PL region of
a given size (see the next result), we believe this is a useful foundation for future 
work that can tune the fixed algorithmic parameters in order 
to allow for a similar statement with an isolated, local PL region of any size.

\begin{theorem} \label{thm-locally-pl-convergence-rate}
    Suppose \eqref{eqn-problem} satisfies assumptions \ref{as:lower-bounded} and \ref{as:locally-lipschitz-cont}.
    Let $C \subset \mathbb{R}^n$ be compact, and let 
    $\mathcal{G} = \sup_{\theta \in C \cup \lbrace \psi : \inf_{\phi \in C} \inlinenorm{\psi - \phi}_2 \leq 10 \rbrace} \inlinenorm{\dot F(\theta)}_2$.
    Suppose $C$ is an isolated, local PL region of $F$ with size $10 + \overline{\alpha}\overline{\delta} \mathcal{G}$. 
    Suppose we solve \eqref{eqn-problem} with algorithm \ref{alg:novel-step-size} initialized at any $\theta_0 \in \mathbb{R}^n$.
    Let $\lbrace \ell_t : t + 1 \in \mathbb{N} \rbrace$ be defined as in \eqref{eqn-iterate-subsequence}.
    Suppose $\exists \tau \in \mathbb{N}\cup\lbrace 0 \rbrace$ such that $\theta_{\ell_\tau} \in C$. 
    Then, either the algorithm terminates in finite time at a stationary point in $C$, or, 
    $\forall t \geq \tau$, $\theta_{\ell_t} \in C$ and 
    \begin{equation} \label{eq-local-pl-rate}
        (F(\theta_{\ell_t}) - \underline{F}) \leq 
        (1 - \rho\underline{\delta}\underline{\alpha}\mu)^{t-\tau} (F(\theta_{\ell_{\tau}}) - \underline{F}),
    \end{equation}
    where $\underline F = \inf_{\theta \in C} F(\theta)$; and $\mu$ is the local PL constant 
    defined in Definition \ref{def-isolated-pl-region}.
\end{theorem}
\begin{proof}
    Let $\mathcal{N} = \lbrace \theta: \mathrm{dist}(\theta, C) \leq 10 + 
    \overline{\alpha} \overline{\delta} \mathcal{G} \rbrace$.
    We first show that for all $k \geq \ell_{\tau}$, $\theta_{k} \in C$.
    The proof is by induction. The base is case is similar to the generalization 
    statement, so we show the generalization only.
    By the induction hypothesis, $\theta_k \in C$. Let $C_k$ be defined as in 
    \eqref{eqn-compactset-outerloop}. 
    Hence, $C_k \subseteq \mathcal{N}$; and, in particular, $\psi_{j_k}^k \in \mathcal{N}$.
    There are two cases. 
    \begin{enumerate}
        \item If $\psi_{j_k}^k \in \mathcal{N} - C$, then $F(\psi_{j_{k}^k}) 
        > \sup_{\theta \in C} F(\theta) \geq F(\theta_k)$. Hence, $\psi_{j_k}^k$ 
        would be rejected by our nonsequential Armijo condition and 
        $\theta_{k+1} = \theta_k \in C$. 
        \item If $\psi_{j_k}^k \in C$, then regardless of whether it is accepted 
        or rejected, $\theta_{k+1} \in C$. 
    \end{enumerate}
    
    We conclude that the iterates remain in $C$ and are thus bounded.
    By Corollary \ref{result-adjustment-parameter-lower-bound},
    $\exists \underline \delta > 0$ such that $\delta_k \geq \underline{\delta}$.
    Using Theorem \ref{result-zoutendjik}, either the algorithm terminates 
    in a finite time at a stationary point that must be in $C$ or 
    $\lim_{k} \inlinenorm{\dot F(\theta_k)}_2 = 0$. 
    In the latter case, by \eqref{eqn-guaranteed-reduction}, the isolated, 
    local PL condition on $C$, and Lemma \ref{result-bounded-step-size},
    \begin{equation}
        F(\theta_{\ell_{t+1}}) < F(\theta_{\ell_t}) 
            - \rho \underline \delta \underline \alpha \mu
            (F(\theta_{\ell_t}) - \underline{F}),~\forall t \geq \tau.
    \end{equation}
    Subtracting $\underline{F}$ from both sides of the inequality 
    and iterating yields the result for $t > \tau$, and when $t = \tau$
    the inequality follows trivially. 
    \qed
\end{proof}

\section{Numerical Experiments} 
\label{sec:experiments}
Here, we seek to compare the reliability of OFFO methods and our method,
and characterize the sensitivity of OFFO methods to their hyperparameters on a set of
quasi-likelihood estimation problems \cite{wedderburn1974quasilikelihood} 
that require numerical integration to evaluate the objective
and with locally Lipschitz continuous gradients. 
In particular, we define reliability and sensitivity in the context of our
experiment through the following questions.
\begin{enumerate}
    \item (Reliability) How often do OFFO methods and our method reach approximate 
    stationary points and/or decrease the objective across their hyperparameter ranges, 
    and problem configuration given a max number of iterations 
    and gradient tolerance stopping condition? 
    \item (Sensitivity) Does the probability that an OFFO method 
    reach an approximate stationary point and/or decrease the objective 
    differ across hyperparameter settings given
    a max number of iterations and gradient tolerance stopping condition?
\end{enumerate}

We seek to answer these questions through a comparison on a set of
quasi-likelihood problems for which the optimization problem takes the form
\begin{equation} \label{eqn-ql-objective}
    \min_{\theta \in \mathbb{R}^n} F(\theta) = 
    \min_{\theta \in \mathbb{R}^n} -\left(\sum_{i=1}^m \int_{0}^{g(x_i^\intercal \theta)} 
    \frac{y_i - \mu}{V(\mu)}d\mu \right),
\end{equation}
for functions $g : \mathbb{R} \to \mathbb{R}$ (i.e., the link function) 
and $V : \mathbb{R} \to \mathbb{R}_{>0}$ (i.e., the variance function)
\cite{wedderburn1974quasilikelihood,mccullagh1989glm}.
In this experiment, we fix the link function to be the logistic function
\begin{equation}
    g(\eta) = \frac{1}{1 + \exp(-\eta)},
\end{equation}
and select from four variance functions motivated by the literature (in all our experiments,
$p = 2.25$ below).
\begin{enumerate}
    \item $V_1(\mu) = 1 + \mu + \sin(2\pi\mu)$ \cite[Section 4: Case 2]{lanteri2023Designing}
    \item $V_2(\mu) = |\mu|^{2p} + 1$ (i.e., a variance stabilizing transformation)
    \item $V_3(\mu) = \exp(|\mu - 1|^{2p})$ \cite[Section 4: Case 1]{lanteri2023Designing}
    \item $V_4(\mu) = \log(|\mu - 1|^{2p} + 1) + 1$.
\end{enumerate}

For each link and variance function combination, we must decide the size of $\theta$ ($n$)
and how many ``observations'' are generated ($m$). We select a combination of these values from 
Table \ref{table:blocking-factor}.
\begin{table}[!h]
    \centering
    \caption{For each link and variance function combination, the number of observations ($m$) 
    and the number of parameters ($n$) are selected from the sets under Blocks.
    }
    \label{table:blocking-factor}
    \begin{tabular}{p{2in} | p{1in}}
        \toprule
        Blocking System & Blocks \\
        \midrule
        Number of observations ($m$) & $\lbrace 100, 1000 \rbrace$ \\
        Number of variables ($n$) & $\lbrace 10, 50, 100 \rbrace$\\
        \bottomrule
    \end{tabular}
\end{table}

Given this selection, to fully instantiate an objective function we must generate data for the problem; 
in particular, we require a ``true'' set of coefficients $\theta^* \in \mathbb{R}^n$, 
a matrix of ``observations'' $X \in \mathbb{R}^{m \times n}$, a vector of ``responses'' 
$Y \in \mathbb{R}^m$, and a list of starting points $\lbrace \theta^{(1)},...,\theta^{(10)} \rbrace$. 
Each of these components have data generating mechanisms outlined in Table \ref{table:data-generation}.

\begin{table}[!h]
    \centering
    \caption{Let $g:\mathbb{R}\to\mathbb{R}$ be the link function, and $V:\mathbb{R}\to\mathbb{R}$ be the
    variance function. Given $n$ and $m$ selected from the values in Table \ref{table:blocking-factor},
    we generate the data for each problem following the procedure outlined in the table.
    For $\mu \in \mathbb{R}^n$, $\Sigma \in \mathbb{R}^{n \times n}$, $\Sigma$ symmetric and p.s.d,
    and $a, b \in \mathbb{R},~b \geq a$,
    symbols $\mathcal{N}(\mu, \Sigma)$, $\mathcal{U}[a, b]$, $\mathrm{Arcsin}[0,1]$ stand for a
    normal distribution with mean $\mu$ and covariance matrix $\Sigma$, a uniform distribution
    over $[a,b]$, and the arcsin distribution on $[0,1]$, respectively. 
    We denote $\prod_{i=1}^n \mathcal{U}[a,b]$ as the product
    distribution (i.e., a vector of size $n$ where each entry is distributed as $\mathcal{U}[a,b]$).
    Finally, for each ``observation'' vector in $X$, the first entry is replaced with $1$ for the intercept.}
    \label{table:data-generation}
    \begin{tabular}{p{1in} | p{3in}} 
        \toprule
        Data & Data Generation Mechanism\\
        \midrule
        True Coefficients & $\mu\in\mathbb{R}^n, \theta^* \in \mathbb{R}^n,~\mu\sim\mathcal{N}(0, I),~\theta^* \sim \mathcal{N}(\mu, I)$\\
        Observations & $\forall i \in \{1,...,m\},~x_i \in \mathbb{R}^n,~x_i \sim \mathcal{N}(0, I)$, normalized by $\sqrt{n-1}$,
        and the first entry replaced with $1$\\
        Error & $\forall i \in \{1,...,m\}, ~\upsilon_i \sim \mathrm{Arcsin}[0,1],~\epsilon_i = (\upsilon_i - .5)/\sqrt{(1/8)}$.\\
        Responses & $\forall i \in \{1,...,m\}, ~y_i = g((\theta^*)^\intercal x_i) + V(g((\theta^*)^\intercal x_i))^{1/2} \epsilon_i$\\
        Starting Point & $\forall i \in \{1,...,10\},~\theta^{(i)} \in \mathbb{R}^n,~\theta^{(i)} \sim \prod_{i=1}^n \mathcal{U}[-10,10]$\\
        \bottomrule
    \end{tabular}
\end{table}

We now list the methods we include in our numerical experiment, 
their hyperparameters, and the allowable range for these hyperparameters
in Table \ref{table:algorithms}.

\begin{table}[!h]
    \centering
    \caption{List of algorithms included in the numerical experiment,
    their hyperparameters, and a list of allowable hyperparameters values.}
    \label{table:algorithms}
    \begin{tabular}{p{1.5in} | p{1in} | p{1.5in}}
        \toprule
        Algorithm & Hyperparameter & Hyperparameter List\\
        \midrule
        Fixed Step Size \cite{bertsekas1999} & Step size & $\lbrace 10^{-4}, 1, 2, 4, 6, 8, 10 \rbrace$ \\
        Diminishing Step Size \cite{patel2024Gradient,bertsekas1999} & Step size sequence and factor & $
        \frac{1}{2^{\lfloor \log_2(k/100) + 1 \rfloor}}$ $\lbrace 10^{-4}, 1, 2, 4, 6, 8, 10 \rbrace$\\
        Barzilai-Borwein \cite{barzilai1988Twopoint} & Long or short step, Initial step size & 
        $\lbrace \mathrm{long}, \mathrm{short} \rbrace$, $\lbrace 10^{-4}, 1, 2, 4, 6, 8, 10 \rbrace$\\
        Lipschitz Approximation \cite{malitsky2023adaptive,malitsky2020Adaptive} & Initial step size & $\lbrace 10^{-4}, 1, 2, 4, 6, 8, 10 \rbrace$\\
        Nesterov's Accelerated GD \cite{li2023Convex} & Step size & $\lbrace 10^{-4}, 1, 2, 4, 6, 8, 10 \rbrace$\\
        Weighted Norm Gradient \cite{wu2020Wngrad} & Initial step size & $\lbrace 10^{-4}, 1, 2, 4, 6, 8, 10 \rbrace$\\
        Our method (algorithm \ref{alg:novel-step-size}) & $\rho, \delta_0, \overline{\delta}$ & $10^{-4}, 1, 1$\\
        \bottomrule
    \end{tabular}
\end{table}

Finally, the experiment is conducted as follows; for each algorithm 
and hyperparameter selection, for each link, variance, and
blocking factor combination, we generate a problem and run the algorithm ten times starting
once from each of the starting points in $\lbrace \theta^{(1)},...,\theta^{(10)} \rbrace$.
Each algorithm is run until either $5{,}000$ iterations are reached, or the gradient falls
below $10^{-3}$. For each specified problem, specified algorithm, and starting point (i.e.,
one observational unit) an observation consists of the following: problem name
(i.e., what variance function is used), $n$, $m$, algorithm name, algorithm hyperparameter(s),
starting point number,
the starting objective value, the ending objective value, the starting gradient value, and
the ending gradient value. 

This experiment was conducted on the Statistics Computing Cluster at UW-Madison, and all
methods and experiments were implemented in Julia-v1.11.
For reproducibility and a more in-depth discussion on our experimental design,
see our experiment library,\footnote{https://github.com/numoptim/experiments-stepsizeselection} and
for implementations of all the methods see our optimization library.\footnote{https://github.com/numoptim/OptimizationMethods.jl}

\subsection{Results and Discussion}
Figure \ref{fig-perc-approx-station} contains plots of
the percentage of observations
that an approximate stationary point\footnote{
The final norm gradient is less than or equal to $10^{-3}$ and was a real 
number (was not NaN).} 
was found as a function of step size and paneled by each problem. 
As our method does not require a user-supplied step size, 
figure \ref{fig-perc-approx-station} shows our method as a horizontal dashed line.
As seen in figure \ref{fig-perc-approx-station}, our method has consistent performance 
across all problems, whereas this is not the case for OFFO methods: many more instances 
of OFFO solve attempts fail to find an approximate stationary point for 
quasi-likelihood problem with the variance function $V_2$. 
This behavior illustrates the sensitivity of choosing OFFO parameters to the problem,
whereas our method does not suffer from such sensitivity concerns.

% % decomposition of objective value and stationary points
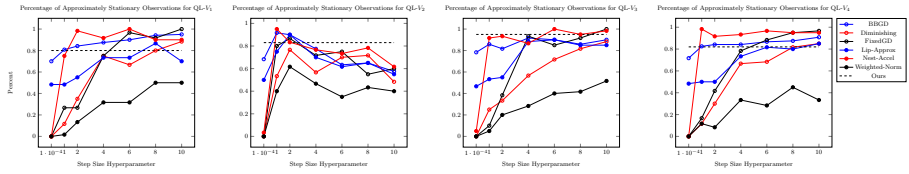
\begin{figure}[h!]
    \centering
    \begin{tikzpicture}[scale=.3]
    % QLLOGCEXP
    \begin{axis}[title=Percentage of Approximately Stationary Observations for QL-$V_1$,
        xlabel = Step Size Hyperparameter,
        ylabel = Percent,
        xtick={0.0001,1,2,4,6,8,10}]
        % BBGD
        \addplot[mark=o, color=blue] 
            table[x="step.size.param", y="QLLOGSIN",col sep=comma]
            {sections/figures/data/perc_other_approx_station_BBGD.csv};

        % DIMINISHING
        \addplot[mark=o, color=red] 
            table[x="step.size.param", y="QLLOGSIN",col sep=comma]
            {sections/figures/data/perc_other_approx_station_DIMINISHING.csv};

        % FIXEDGD
        \addplot[mark=o, color=black] 
            table[x="step.size.param", y="QLLOGSIN",col sep=comma]
            {sections/figures/data/perc_other_approx_station_FIXEDGD.csv};

        % LIPAPPROX
        \addplot[mark=*, color=blue] 
            table[x="step.size.param", y="QLLOGSIN",col sep=comma]
            {sections/figures/data/perc_other_approx_station_LIPAPPROX.csv};

        % NESTACCEL
        \addplot[mark=*, color=red] 
            table[x="step.size.param", y="QLLOGSIN",col sep=comma]
            {sections/figures/data/perc_other_approx_station_NESTACCEL.csv};

        % WEIGHTNORMGD
        \addplot[mark=*, color=black] 
            table[x="step.size.param", y="QLLOGSIN",col sep=comma]
            {sections/figures/data/perc_other_approx_station_WEIGHTNORMGD.csv};
        
        \addplot [domain=1e-4:10,
            thick,
            dashed,
            black,
            ]{.8};
    \end{axis}
\end{tikzpicture}
    \begin{tikzpicture}[scale=.3]
    \begin{axis}[title=Percentage of Approximately Stationary Observations for QL-$V_2$,
        xlabel = Step Size Hyperparameter,
        xtick={0.0001,1,2,4,6,8,10}]
        % BBGD
        \addplot[mark=o, color=blue] 
            table[x="step.size.param", y="QLLOGMON",col sep=comma]
            {sections/figures/data/perc_other_approx_station_BBGD.csv};

        % DIMINISHING
        \addplot[mark=o, color=red] 
            table[x="step.size.param", y="QLLOGMON",col sep=comma]
            {sections/figures/data/perc_other_approx_station_DIMINISHING.csv};

        % FIXEDGD
        \addplot[mark=o, color=black] 
            table[x="step.size.param", y="QLLOGMON",col sep=comma]
            {sections/figures/data/perc_other_approx_station_FIXEDGD.csv};

        % LIPAPPROX
        \addplot[mark=*, color=blue] 
            table[x="step.size.param", y="QLLOGMON",col sep=comma]
            {sections/figures/data/perc_other_approx_station_LIPAPPROX.csv};

        % NESTACCEL
        \addplot[mark=*, color=red] 
            table[x="step.size.param", y="QLLOGMON",col sep=comma]
            {sections/figures/data/perc_other_approx_station_NESTACCEL.csv};

        % WEIGHTNORMGD
        \addplot[mark=*, color=black] 
            table[x="step.size.param", y="QLLOGMON",col sep=comma]
            {sections/figures/data/perc_other_approx_station_WEIGHTNORMGD.csv};

        \addplot [domain=1e-4:10,
            thick,
            dashed,
            black,
            ]{.83};
    \end{axis}
\end{tikzpicture}
    \begin{tikzpicture}[scale = .3]
    % QLLOGCEXP
    \begin{axis}[title=Percentage of Approximately Stationary Observations for QL-$V_3$,
        xlabel = Step Size Hyperparameter,
        scaled x ticks=false,
        xtick={0.0001,1,2,4,6,8,10}]
        % BBGD
        \addplot[mark=o, color=blue] 
            table[x="step.size.param", y="QLLOGCEXP",col sep=comma]
            {sections/figures/data/perc_other_approx_station_BBGD.csv};

        % DIMINISHING
        \addplot[mark=o, color=red] 
            table[x="step.size.param", y="QLLOGCEXP",col sep=comma]
            {sections/figures/data/perc_other_approx_station_DIMINISHING.csv};

        % FIXEDGD
        \addplot[mark=o, color=black] 
            table[x="step.size.param", y="QLLOGCEXP",col sep=comma]
            {sections/figures/data/perc_other_approx_station_FIXEDGD.csv};

        % LIPAPPROX
        \addplot[mark=*, color=blue] 
            table[x="step.size.param", y="QLLOGCEXP",col sep=comma]
            {sections/figures/data/perc_other_approx_station_LIPAPPROX.csv};

        % NESTACCEL
        \addplot[mark=*, color=red] 
            table[x="step.size.param", y="QLLOGCEXP",col sep=comma]
            {sections/figures/data/perc_other_approx_station_NESTACCEL.csv};

        % WEIGHTNORMGD
        \addplot[mark=*, color=black] 
            table[x="step.size.param", y="QLLOGCEXP",col sep=comma]
            {sections/figures/data/perc_other_approx_station_WEIGHTNORMGD.csv};

        \addplot [domain=1e-4:10,
        thick,
        dashed,
        black,
        ]{.95};
    \end{axis}
\end{tikzpicture}
    \begin{tikzpicture}[scale=.3]
    % QLLOGCEXP
    \begin{axis}[legend pos=outer north east,
        title=Percentage of Approximately Stationary Observations for QL-$V_4$,
        xlabel = Step Size Hyperparameter,
        xtick={0.0001,1,2,4,6,8,10}]
        % BBGD
        \addplot[mark=o, color=blue] 
            table[x="step.size.param", y="QLLOGCLOG",col sep=comma]
            {sections/figures/data/perc_other_approx_station_BBGD.csv};
        \addlegendentry{BBGD}

        % DIMINISHING
        \addplot[mark=o, color=red] 
            table[x="step.size.param", y="QLLOGCLOG",col sep=comma]
            {sections/figures/data/perc_other_approx_station_DIMINISHING.csv};
        \addlegendentry{Diminishing}

        % FIXEDGD
        \addplot[mark=o, color=black] 
            table[x="step.size.param", y="QLLOGCLOG",col sep=comma]
            {sections/figures/data/perc_other_approx_station_FIXEDGD.csv};
        \addlegendentry{FixedGD}

        % LIPAPPROX
        \addplot[mark=*, color=blue] 
            table[x="step.size.param", y="QLLOGCLOG",col sep=comma]
            {sections/figures/data/perc_other_approx_station_LIPAPPROX.csv};
        \addlegendentry{Lip-Approx}

        % NESTACCEL
        \addplot[mark=*, color=red] 
            table[x="step.size.param", y="QLLOGCLOG",col sep=comma]
            {sections/figures/data/perc_other_approx_station_NESTACCEL.csv};
        \addlegendentry{Nest-Accel}

        % WEIGHTNORMGD
        \addplot[mark=*, color=black] 
            table[x="step.size.param", y="QLLOGCLOG",col sep=comma]
            {sections/figures/data/perc_other_approx_station_WEIGHTNORMGD.csv};
        \addlegendentry{Weighted-Norm}

        \addplot [domain=1e-4:10,
        thick,
        dashed,
        black,
        ]{.82};
        \addlegendentry{Ours}
    \end{axis}
\end{tikzpicture}
    \caption{
        Percentage of observations
        where the norm gradient at the terminal iterate was 
        a real number (was not NaN), and was smaller than or equal to $10^{-3}$ 
        for each algorithm, for each of their hyperparameter settings, paneled by
        the quasi-likelihood problem. 
        \label{fig-perc-approx-station}
    }
\end{figure}

Figure \ref{fig-perc-achieved-descent} shows the fraction of observations
where the terminal iterate's objective function value is smaller than the initial iterate's 
objective function value as a function of the step size and paneled by the quasi-likelihood 
problem. 
In each panel of figure \ref{fig-perc-achieved-descent}, our method appears as a horizontal 
dashed line as our method does not require a user-specified step size. 
Furthermore, as we would expect from our descent condition, our method consistently improves 
on the optimality gap. 
On the other hand, OFFO methods have mixed performance in improving the objective function,
which lends evidence that the anti-convergence behavior that we present in Section \ref{sec:counter}
is not solely of theoretical interest. 
In other words, OFFO methods seem to be less reliable for non-convex problems. 

\begin{figure}[h!]
    \centering
    \subfigure{\begin{tikzpicture}[scale=.3]
    % QLLOGCEXP
    \begin{axis}[title=Percentage of Observations that Achieved Descent for QL-$V_1$,
        xlabel = Step Size Hyperparameter,
        ylabel = Percent,
        xtick={0.0001,1,2,4,6,8,10}]
        % BBGD
        \addplot[mark=o, color=blue] 
            table[x="step.size.param", y="QLLOGSIN",col sep=comma]
            {sections/figures/data/perc_other_achieved_descent_BBGD.csv};

        % DIMINISHING
        \addplot[mark=o, color=red] 
            table[x="step.size.param", y="QLLOGSIN",col sep=comma]
            {sections/figures/data/perc_other_achieved_descent_DIMINISHING.csv};

        % FIXEDGD
        \addplot[mark=o, color=black] 
            table[x="step.size.param", y="QLLOGSIN",col sep=comma]
            {sections/figures/data/perc_other_achieved_descent_FIXEDGD.csv};

        % LIPAPPROX
        \addplot[mark=*, color=blue] 
            table[x="step.size.param", y="QLLOGSIN",col sep=comma]
            {sections/figures/data/perc_other_achieved_descent_LIPAPPROX.csv};

        % NESTACCEL
        \addplot[mark=*, color=red] 
            table[x="step.size.param", y="QLLOGSIN",col sep=comma]
            {sections/figures/data/perc_other_achieved_descent_NESTACCEL.csv};

        % WEIGHTNORMGD
        \addplot[mark=*, color=black] 
            table[x="step.size.param", y="QLLOGSIN",col sep=comma]
            {sections/figures/data/perc_other_achieved_descent_WEIGHTNORMGD.csv};
        
        \addplot [domain=1e-4:10,
            thick,
            dashed,
            black,
            ]{1.0};
    \end{axis}
\end{tikzpicture}}
    \subfigure{\begin{tikzpicture}[scale=.3]
    \begin{axis}[title=Percentage of Observations that Achieved Descent for QL-$V_2$,
        xlabel = Step Size Hyperparameter,
        xtick={0.0001,1,2,4,6,8,10}]
        % BBGD
        \addplot[mark=o, color=blue] 
            table[x="step.size.param", y="QLLOGMON",col sep=comma]
            {sections/figures/data/perc_other_achieved_descent_BBGD.csv};

        % DIMINISHING
        \addplot[mark=o, color=red] 
            table[x="step.size.param", y="QLLOGMON",col sep=comma]
            {sections/figures/data/perc_other_achieved_descent_DIMINISHING.csv};

        % FIXEDGD
        \addplot[mark=o, color=black] 
            table[x="step.size.param", y="QLLOGMON",col sep=comma]
            {sections/figures/data/perc_other_achieved_descent_FIXEDGD.csv};

        % LIPAPPROX
        \addplot[mark=*, color=blue] 
            table[x="step.size.param", y="QLLOGMON",col sep=comma]
            {sections/figures/data/perc_other_achieved_descent_LIPAPPROX.csv};

        % NESTACCEL
        \addplot[mark=*, color=red] 
            table[x="step.size.param", y="QLLOGMON",col sep=comma]
            {sections/figures/data/perc_other_achieved_descent_NESTACCEL.csv};

        % WEIGHTNORMGD
        \addplot[mark=*, color=black] 
            table[x="step.size.param", y="QLLOGMON",col sep=comma]
            {sections/figures/data/perc_other_achieved_descent_WEIGHTNORMGD.csv};

        \addplot [domain=1e-4:10,
            thick,
            dashed,
            black,
            ]{1.0};
    \end{axis}
\end{tikzpicture}}
    \subfigure{% Date: 2025/02/28
% Author: Christian Varner
% Purpose: Display the percentage of times each algorithm achieved descent

\begin{tikzpicture}[scale = .3]
    % QLLOGCEXP
    \begin{axis}[title=Percentage of Observations that Achieved Descent for QL-$V_3$,
        xlabel = Step Size Hyperparameter,
        scaled x ticks=false,
        xtick={0.0001,1,2,4,6,8,10}]
        % BBGD
        \addplot[mark=o, color=blue, xtick={0.0001,1,2,4,6,8,10}] 
            table[x="step.size.param", y="QLLOGCEXP",col sep=comma]
            {sections/figures/data/perc_other_achieved_descent_BBGD.csv};

        % DIMINISHING
        \addplot[mark=o, color=red] 
            table[x="step.size.param", y="QLLOGCEXP",col sep=comma]
            {sections/figures/data/perc_other_achieved_descent_DIMINISHING.csv};

        % FIXEDGD
        \addplot[mark=o, color=black] 
            table[x="step.size.param", y="QLLOGCEXP",col sep=comma]
            {sections/figures/data/perc_other_achieved_descent_FIXEDGD.csv};

        % LIPAPPROX
        \addplot[mark=*, color=blue] 
            table[x="step.size.param", y="QLLOGCEXP",col sep=comma]
            {sections/figures/data/perc_other_achieved_descent_LIPAPPROX.csv};

        % NESTACCEL
        \addplot[mark=*, color=red] 
            table[x="step.size.param", y="QLLOGCEXP",col sep=comma]
            {sections/figures/data/perc_other_achieved_descent_NESTACCEL.csv};

        % WEIGHTNORMGD
        \addplot[mark=*, color=black] 
            table[x="step.size.param", y="QLLOGCEXP",col sep=comma]
            {sections/figures/data/perc_other_achieved_descent_WEIGHTNORMGD.csv};

        \addplot [domain=1e-4:10,
        thick,
        dashed,
        black,
        ]{1.0};
    \end{axis}
\end{tikzpicture}}
    \subfigure{\begin{tikzpicture}[scale=.3]
    % QLLOGCEXP
    \begin{axis}[legend pos=outer north east,
        title=Percentage of Observations that Achieved Descent for QL-$V_4$,
        xlabel = Step Size Hyperparameter,
        xtick={0.0001,1,2,4,6,8,10}]
        % BBGD
        \addplot[mark=o, color=blue] 
            table[x="step.size.param", y="QLLOGCLOG",col sep=comma]
            {sections/figures/data/perc_other_achieved_descent_BBGD.csv};
        \addlegendentry{BBGD}

        % DIMINISHING
        \addplot[mark=o, color=red] 
            table[x="step.size.param", y="QLLOGCLOG",col sep=comma]
            {sections/figures/data/perc_other_achieved_descent_DIMINISHING.csv};
        \addlegendentry{Diminishing}

        % FIXEDGD
        \addplot[mark=o, color=black] 
            table[x="step.size.param", y="QLLOGCLOG",col sep=comma]
            {sections/figures/data/perc_other_achieved_descent_FIXEDGD.csv};
        \addlegendentry{FixedGD}

        % LIPAPPROX
        \addplot[mark=*, color=blue] 
            table[x="step.size.param", y="QLLOGCLOG",col sep=comma]
            {sections/figures/data/perc_other_achieved_descent_LIPAPPROX.csv};
        \addlegendentry{Lip-Approx}

        % NESTACCEL
        \addplot[mark=*, color=red] 
            table[x="step.size.param", y="QLLOGCLOG",col sep=comma]
            {sections/figures/data/perc_other_achieved_descent_NESTACCEL.csv};
        \addlegendentry{Nest-Accel}

        % WEIGHTNORMGD
        \addplot[mark=*, color=black] 
            table[x="step.size.param", y="QLLOGCLOG",col sep=comma]
            {sections/figures/data/perc_other_achieved_descent_WEIGHTNORMGD.csv};
        \addlegendentry{Weighted-Norm}

        \addplot [domain=1e-4:10,
        thick,
        dashed,
        black,
        ]{1.0};
    \end{axis}
\end{tikzpicture}}
    \caption{
        After filtering out observations were the starting or ending objective value
        was NaN, we plot the percentage of remaining
        observations where the final objective value was smaller 
        than the starting objective value for 
        each algorithm, across their hyperparameter ranges, paneled by
        quasi-likelihood problem.
        \label{fig-perc-achieved-descent}
    }
\end{figure}
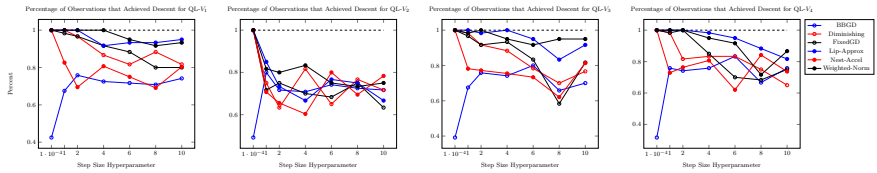

\section{Conclusion}            
\label{sec:conclusion}
In this work, we consider optimization problems with expensive objective function evaluations 
and relatively inexpensive gradient evaluations arising in a variety of applications.
We first consider OFFO methods for optimization problems with this computing pattern, 
but we show that these methods experience anti-convergence behavior 
in which they can generate an optimality gap that diverges on the problem class. 

We propose and analyze a novel first-order algorithm motivated by the limitation of OFFO 
methods with three critical features: event-based objective function evaluations;
a non-sequential Armijo condition; and a step size strategy that exploits the 
properties induced by our event-based objective function evaluations. 
We show that our method is more successful in finding approximate stationary points and 
reducing the objective function in comparison to OFFO methods on a set of 
quasi-likelihood objectives.

In subsequent work, we hope to generalize this framework,
extend it to the stochastic setting (both in gradient and objective cases), and
to other optimization schemes.

%\begin{acknowledgements}
%\end{acknowledgements}

\section*{Conflict of interest}

The authors declare that they have no conflict of interest.

\bibliographystyle{spmpsci}      % mathematics and physical sciences
\bibliography{stepsizeselection.bib}   % name your BibTeX data base

\end{document}